\documentclass[11pt]{amsart} 
\usepackage[left=3cm,right=3cm,top=3cm,bottom=3cm]{geometry} 
\usepackage{amsmath,amssymb} 
\usepackage{amsmath, amsthm,amssymb,amsfonts,latexsym}
\usepackage{amsthm}
\usepackage{tikz}
\usepackage{lipsum}
\usepackage{color}

\usepackage{hyperref}
\usepackage[capitalise]{cleveref}
\usetikzlibrary{patterns}
\usetikzlibrary{intersections}
\usetikzlibrary{arrows,positioning}
\usepackage{graphicx}
\usepackage{pgfplots}
 \pgfplotsset{compat=1.14}
\usetikzlibrary{plotmarks}
  \usepgfplotslibrary{patchplots}
  \usepackage{grffile}
\usepgfplotslibrary{fillbetween}
\usepackage{float}
\usepackage{caption}
\usepackage{subcaption}

\newcommand*{\rom}[1]{\expandafter\@slowromancap\romannumeral #1@}

\newcommand{\gettikzxy}[3]{%
  \tikz@scan@one@point\pgfutil@firstofone#1\relax
  \edef#2{\the\pgf@x}%
  \edef#3{\the\pgf@y}%
}

\def\namedlabel#1#2{\begingroup
	#2%
	\def\@currentlabel{#2}%
	\phantomsection\label{#1}\endgroup
}

\newcommand{\sB}{\mathcal B}
\newcommand{\sC}{\mathcal P}
\newcommand{\sD}{\mathcal D}

\newcommand{\sF}{\mathcal F}

\newcommand{\sI}{\mathcal I}

\newcommand{\sM}{\mathcal M}

\newcommand{\sP}{\mathcal P}

\newcommand{\sT}{\mathcal T}

\newcommand{\sX}{\mathcal X}
\newcommand{\sY}{\mathcal Y}

\newcommand{\R}{\mathbb R}

\newcommand{\argmin}{\mbox{argmin}}
\newcommand{\argmax}{\mbox{argmax}}

\newcommand{\Leb}{\mbox{Leb}}
\newcommand{\supp}{\mbox{supp}}

\newtheorem{theorem}{Theorem}[section]

\newtheorem{proposition}{Proposition}[section]

\newtheorem{lemma}{Lemma}[section]

\newtheorem{corollary}{Corollary}[section]
\newtheorem{remark}{Remark}[section]

\newtheorem{defn}{Definition}[section]

\usepackage[foot]{amsaddr}
\makeatletter
\renewcommand{\email}[2][]{%
	\ifx\emails\@empty\relax\else{\g@addto@macro\emails{,\space}}\fi%
	\@ifnotempty{#1}{\g@addto@macro\emails{\textrm{(#1)}\space}}%
	\g@addto@macro\emails{#2}%
}

\makeatother
\makeatletter
\@namedef{subjclassname@2020}{%
	\textup{2020} Mathematics Subject Classification}
\makeatother
\numberwithin{equation}{section}

\begin{document}
\title[Generalized shadow measure and WOT]{Generalizing Super/Sub MOT using weak $L^1$ transport}

\author{Erhan Bayraktar}
\address{Department of Mathematics, University of Michigan}
\email{erhan@umich.edu}

\author{Dominykas Norgilas}
\email{dnorgil@ncsu.edu}
\address{Department of Mathematics, North Carolina State University}
\thanks{E. Bayraktar is partially supported by the National Science Foundation under grant  DMS-2106556 and by the Susan M. Smith chair.} 



\begin{abstract}
In this article we revisit the weak optimal transport (WOT) problem, introduced by Gozlan, Roberto, Samson and Tetali \cite{gozlan2017kantorovich}. We work on the real line, with barycentric cost functions, and as our first result give the following characterization of the set of optimal couplings for two probability measures $\mu$ and $\nu$: every optimizer couples the left tails of $\mu$ and $\nu$ using a submartingale, the right tails using a supermartingale, while the central region is coupled using a martingale.

We then consider a constrained optimal transport problem, where admissible transport plans are only those that are optimal for the WOT problem with $L^1$ costs. The constrained problem generalizes the (sub/super-) martingale optimal transport problems, studied by Beiglb\"ock and Juillet \cite{BeiglbockJuillet:16}, and Nutz and Stebegg \cite{NutzStebegg.18} among others. 

Finally we introduce a generalized \textit{shadow measure} and establish its connection to the WOT. This extends and generalizes the results obtained in (sub/super-) martingale settings.
\end{abstract}

\maketitle
\section{Introduction}
For two probability measures on Polish spaces $\mathcal X,\mathcal Y$, denoted by $\mu\in\sP(\mathcal X)$ and $\nu\in\sP(\mathcal Y)$, and a cost function $c:\mathcal X\times\sP(\mathcal Y)\to\R$, the \textit{weak} optimal transport (WOT) problem is to 
\begin{equation}\label{eq:mainProblem}
\textrm{minimize}\quad \Pi(\mu,\nu)\ni\pi\mapsto\int_{\mathcal X} c(x,\pi_x)\mu(dx),
\end{equation}
where $\Pi(\mu,\nu)$ is the set of probability measures on $\mathcal X\times\mathcal Y$ with first and second marginals $\mu$ and $\nu$, respectively, and, for each $\pi\in\Pi(\mu,\nu)$, $(\pi_x)_{x\in\mathcal X}$ is the $\mu$-a.s. unique family of probability measures such that $\pi(dx,dy)=\mu(dx)\pi_x(dy)$.

This type of transportation cost first appeared in the works of Marton \cite{marton1996measure, Marton1996BoundingB} and Talagrand \cite{talagrand1995concentration, talagrand1996new}, where the authors studied the concentration of measure phenomenon via transport-entropy inequalities. The rather recent works of Gozlan et al. \cite{gozlan2017kantorovich, gozlan2018characterization} (see also Aliberti et al. \cite{alibert2019new}), on the other hand, provide a first systematic study of this problem, including general definitions and Kantorovich-type duality results. \cite{gozlan2017kantorovich} stimulated several research groups to further study \eqref{eq:mainProblem} in various contexts (see, for example, \cite{alfonsi2017sampling, fathi2018curvature, gozlan2020mixture, shu2020hopf, shu2018hamilton}). Together with duality, essentially all the fundamental theoretical results of the classic optimal transport (existence of optimizers, cyclical monotonicity, stability) were extended to the weak setting, see \cite{backhoff2019existence, backhoff2020martingale, backhoff2022stability, gozlan2020mixture}. Consult also \cite{backhoff2022applications} for a list of applications and connections of weak optimal transport to other fields (including Schr\"odinger problem, Brenier-Strassen theorem, optimal mechanism design, linear transfers and semimartingale transport).

\subsection{Our contribution} In this paper we investigate \eqref{eq:mainProblem} on the real line (so that $\sX=\sY=\R$, and $\mu,\nu\in\sP(\R)$), and with \textit{barycentric} costs. In particular, for a convex $h:\R\to\R$, the goal is to
\begin{equation}\label{eq:mainProblemBary}
    minimize\quad \pi\mapsto\int_{\R} h\left(x-\int_\R y\pi_x(dy)\right)\mu(dx)\quad over\quad\Pi(\mu,\nu).
\end{equation}
This case (in one and higher dimensions) has received particular attention, see \cite{alfonsi2017sampling, backhoff2019existence, fathi2018curvature, gozlan2020mixture, gozlan2018characterization, gozlan2017kantorovich, shu2020hopf, shu2018hamilton, backhoff2020adapted, backhoff2020weak}.

\textit{Section \ref{sec:wot}: WOT with $L^1$ costs.} We first study \eqref{eq:mainProblemBary} with $h(\cdot)=\lvert \cdot\lvert$. In this case we show that $\pi\in\Pi(\mu,\nu)$ is an optimizer for \eqref{eq:mainProblemBary} \textit{if and only if} it has a rather simple structure (see Theorem \ref{thm:coupling}): there exists $x^-\leq x^+$ (which are explicitly given in terms of the so-called potential functions of the underlying measures (see Definition \ref{eq:x-x+})) such that
\begin{align}\label{eq:L1optimal}
\overline{\pi_x}\geq x\quad&\textrm{and}\quad\supp(\pi_x)\subseteq(-\infty,x^-],\quad\textrm{for all }x\leq x^-,\nonumber\\
\overline{\pi_x}= x\quad&\textrm{and}\quad\supp(\pi_x)\subseteq[x^-,x^+],\quad\textrm{for all }x^-<x<x^+,\\
\overline{\pi_x}\leq x\quad&\textrm{and}\quad\supp(\pi_x)\subseteq[x^+,\infty),\quad\textrm{for all }x^+\leq x,\nonumber
\end{align}
where $\overline\eta$ and $\supp(\eta)$ denote the mean and the support of a measure $\eta$. \eqref{eq:L1optimal} essentially means that $\pi$ is optimal if and only if it couples the left (resp. right) tails of $\mu$ and $\nu$ using a submartingale (resp. supermartingale), while the restrictions of both measures to the central region $(x^-,x^+)$ are coupled via martingale. 

We next investigate the implications of the decomposition \eqref{eq:L1optimal} (obtained for the absolute value cost function) to the more general problem \eqref{eq:mainProblemBary} when $h$ is an arbitrary convex function. In particular, we show that even in this case any optimizer of \eqref{eq:mainProblemBary} admits decomposition \eqref{eq:L1optimal}; see Corollary \ref{cor:valueGeneral}. From this we conclude that the only interesting cases of the problem \eqref{eq:L1optimal} essentially are those with marginals that are either in convex increasing or convex decreasing order (i.e., either $\mu\leq_{ci}\nu$ or $\mu\leq_{cd}\nu$, so that the marginals can be coupled using submartingale or supermartingale, respectively). 

The proof of the above result (i.e., that the decomposition \eqref{eq:L1optimal} is still valid for the optimizers of \eqref{eq:mainProblemBary} with general $h$) relies on the following equivalence obtained in \cite{gozlan2017kantorovich}:
\begin{equation}\label{eq:equivalence}
\inf_{\pi\in\Pi(\mu,\nu)}\int_{\R} h\left(x-\int_\R y\pi_x(dy)\right)\mu(dx)=\inf_{\mu^*\leq_c\nu}\inf_{\pi^*\in\Pi(\mu,\mu^*)}\int_{\R\times\R}h(x-y)\pi^*(dx,dy).
\end{equation}
(Here $\leq_c$ denotes the convex order, i.e., if $\eta\leq_c\chi$, then there exists a martingale coupling of $\eta$ and $\chi$ and vice versa; see section \ref{sec:prelims} for precise definitions.) While there are many optimal couplings for \eqref{eq:mainProblemBary}, there is only one $\mu^*$ that minimizes the right hand side of \eqref{eq:equivalence}. In particular, there exists a $\mu$-a.s. unique map $T^*:\R\to\R$, such that the push-forward measure $\mu^*=T^*(\mu)$ is optimal; see \cite{gozlan2018characterization, alfonsi2017sampling, shu2018hamilton, backhoff2020weak}. An important feature of $\mu^*$ (or, equivalently, of $T^*$) is that it does not depend on the cost function $h$ (we note that this feature fails in higher dimensions, see \cite[Example 2.4]{alfonsi2017sampling}). Using this, together with the fact that \eqref{eq:L1optimal} is valid when $h(\cdot)=\lvert\cdot\lvert$, we obtain that (see Proposition \ref{prop:optimalT})
\begin{align}\label{eq:Toptimal}
x\mapsto (T^*(x)-x)\quad \textrm{ is positive and decreasing on}\quad (-\infty,x^-),\nonumber\\
T^*(x)=x \quad\textrm{for all}\quad x\in[x^-,x^+],\\
x\mapsto (T^*(x)-x)\quad \textrm{ is negative and decreasing on}\quad (x^+,\infty).\nonumber
\end{align}
With \eqref{eq:Toptimal} (and also the equality \eqref{eq:equivalence}) at hand, we then come back to the general problem \eqref{eq:mainProblemBary} and (by borrowing some results from \cite{gozlan2020mixture}) deduce the decomposition \eqref{eq:L1optimal}.

\textit{Section \ref{sec:cot}: Constrained Optimal Transport.} For a given cost function $c:\R^2\to\R$, and a pair of marginals $(\mu,\nu)$, the main goal of the classical optimal transport theory is to minimize $\pi\mapsto\int_{\R^2}cd\pi$ over the set of couplings $\Pi(\mu,\nu)$. In this section the goal is similar: we aim to 
\begin{equation}\label{eq:constrainedOTintro}
minimize\quad\pi\mapsto \int_{\R^2}c(x,y)\pi(dx,dy)\quad over\quad \Pi^*(\mu,\nu),
\end{equation}
where $\Pi^*(\mu,\nu)\subseteq \Pi(\mu,\nu)$ is the set of couplings that satisfy \eqref{eq:L1optimal}. In other words, leveraging the results of section \ref{sec:wot}, we have that the only admissible couplings (for \eqref{eq:constrainedOTintro}) are those that solve $\eqref{eq:mainProblemBary}$ with $h(\cdot)=\lvert\cdot\lvert$.

The problem \eqref{eq:constrainedOTintro} generalizes the martingale and supermartingale optimal transport problems that were explicitly studied for various specific cost functions; see, for example, \cite{BeiglbockJuillet:16, NutzStebegg.18, HobsonKlimmek, HobsonNeuberger, beiglbock2021shadow, BayDengNorgilas2}. In particular, if $\mu\leq_c\nu$, then $\Pi^*(\mu,\nu)$ reduces to the set of martingale couplings, while if $\mu\leq_{cd}\nu$ (resp. $\mu\leq_{ci}\nu$), then $\Pi^*(\mu,\nu)$ corresponds to the set of supermartingale (resp. submartingale) couplings of $\mu$ and $\nu$.

The set $\Pi^*(\mu,\nu)$ is convex and compact (see Lemma \ref{lem:Pi*properties}) and, in particular, the problem \eqref{eq:constrainedOTintro} admits an optimizer (see Corollary \ref{cor:cotMinimizer}). While the structure of the optimizers, in general, depends on the cost function, (by considering the quadratic cost) we are able to identify (among all elements of $\Pi^*(\mu,\nu)$) two canonical couplings, in the sense that they produce the smallest and largest correlation, respectively, of random variables $X\sim\mu$ and $Y\sim\nu$ (see Proposition \ref{prop:cotCOV}).

\textit{Section \ref{sec:generalShadow}: Generalized shadow measure.} As our last contribution we provide a systematic way to build elements of $\Pi^*(\mu,\nu)$, i.e., couplings of $(\mu,\nu)$ that optimizes \eqref{eq:mainProblemBary} with $h(\cdot)=\lvert\cdot\lvert$. Our method relies on the notion of the \textit{generalized shadow measure of $\mu$ in $\nu$}, denoted by $S^\nu(\mu)$, which was introduced in the martingale and supermartingale settings in \cite{BeiglbockJuillet:16} and \cite{NutzStebegg.18}, respectively, and further studied in \cite{beiglbock2022potential, beiglbock2021shadow, BayDengNorgilas, BayDengNorgilas2, hobson2019left}.

Consider two measures on the real line, $\mu$ and $\nu$, such that $\mu(\R)\leq \nu(\R)$. Let $\mathcal{T}_{\mu,\nu}=\{\theta:\theta(\R)=\mu(\R),~\theta\leq \nu\}$ be the possible target laws of $\mu$ in $\nu$. In the case we can embed $\mu$ in $\nu$ using a martingale (resp. supermartingale), \cite{BeiglbockJuillet:16} (resp. \cite{NutzStebegg.18}) showed that there exists a canonical choice for the target law in $\mathcal T_{\mu,\nu}$: the shadow measure $S^\nu(\mu)$ is the smallest element of $\mathcal T_{\mu,\nu}$ with respect to convex order $\leq_c$ (resp, convex decreasing order $\leq_{cd}$). In both cases, $S^\mu(\nu)$ can be uniquely identified as a second (distributional) derivative of a certain convex  function. In particular, there exists $s_{\mu,\nu}:\R\to\R$ (that depends on $\mu$ and $\nu$ through the potential functions of these measures) that is convex, and $s''_{\mu,\nu}(dx)=S^\nu(\mu)(dx)$.

From the results of \cite{beiglbock2022potential} and \cite{BayDengNorgilas} we then have an important observation: the same convex $s_{\mu,\nu}$ works (i.e., it identifies the shadow measure) in the martingale and supermartingale settings. In this paper we show that, in fact, the convexity of $s_{\mu,\nu}$ prevails for arbitrary $(\mu,\nu)$ (and not only for those measures that can be joined via martingale or supermartingale), and thus the second derivative of $s_{\mu,\nu}$ still corresponds to a measure - this is precisely our generalized shadow measure $S^\nu(\mu)$ (see Lemma \ref{lem:shadow_potential}).

While Lemma \ref{lem:shadow_potential} identifies $s''_{\mu,\nu}=S^\nu(\mu)$ as an element of $\mathcal{T}_{\mu,\nu}$, in Theorem \ref{thm:shadow_minimalityWOT} we show that $S^\nu(\mu)$ is a canonical element with respect to WOT and convex order. In particular, 
\begin{equation}\label{eq:shadowMin}
S^\nu(\mu)~minimizes\quad \theta\mapsto \inf_{\Pi(\mu,\theta)}\int_\R\left\lvert x-\int_\R y\pi_x(dy)\right\lvert\mu(dx)\quad over\quad \mathcal T_{\mu,\nu}.
\end{equation}
However, there could be many minimizers $\theta^*$ of \eqref{eq:shadowMin}. Nevertheless, they all have the same mean, while the shadow measure has the smallest variance. In other words, $S^\nu(\mu)\leq_c\theta^*$ for any other minimizer $\theta^*$.

Finally, we follow the approaches of \cite{beiglbock2021shadow} and \cite{BayDengNorgilas2}, and, using the shadow measure, construct a large family of couplings within $\Pi^*(\mu,\nu)$, the so-called \textit{shadow couplings}. In particular, let $\hat\mu\in\sP([0,1]\times\R)$ be an element of $\Pi(\Leb_{[0,1]},\mu)$, i.e., some coupling of the Lebesgue measure on $[0,1]$ and $\mu$ (in the language of \cite{beiglbock2021shadow}, $\hat\mu$ is a \textit{lift} of $\mu$). Then each $\hat\mu$ induces a parametrization $(\hat\mu_{[0,u]})_{0\leq u\leq 1}$ of $\mu$, where $\hat\mu_{[0,u]}(dx):=\int^u_0\hat\mu(dv,dx)$: we have that $\hat\mu_{[0,u]}(A)\leq \hat\mu_{[0,v]}(A)\leq\hat\mu_{[0,1]}(A)=\mu(A)$ for all Borel $B\subseteq\R$ and $0\leq u\leq v\leq1$. We then show that (see Theorem \ref{thm:shadowCouplings}), for each lift $\hat\mu$,
\begin{align*}
  &there~exists~the~unique~measure~\hat\pi~ on~ [0,1]\times\R\times\R~ that~couples~ \hat\mu~ and ~\nu,\\
  &and~sends~each~\hat\mu_{[0,u]}~to~the~shadow~measure~S^\nu(\hat\mu_{[0,u]}).
\end{align*}

By integrating out the Lebesgue measure we then have that $\int^1_0d\hat\pi$ corresponds to a coupling of $\mu$ and $\nu$, and which, in particular, belongs to $\Pi^*(\mu,\nu)$. At least in the martingale setting, \cite{beiglbock2021shadow} showed that this type of couplings appear as optimizers to a certain weak optimal transport problems.

\section{Measures and Convex order}
\label{sec:prelims}
For $k\geq 1$, let $\sM^k$ (respectively $\sP^k$) be the set of (Borel) measures (respectively probability measures) on $\R^k$ with finite total mass and finite first moment. When $k=1$, we will use the notation $\sM=\sM^1$ (resp. $\sP=\sP^1$), so that if $\eta\in\sM$ (resp. $\eta\in\sP$), then $\eta(\R)<\infty$ (resp. $\eta(\R)=1$) and $\int_\R\lvert x\lvert\eta(dx)<\infty$. Given a measure $\eta\in\sM$ (not necessarily a probability measure), define $\overline{\eta} = \int_\R x \eta(dx)$ to be the first moment (or mean) of $\eta$ (and then $\overline{\eta}/\eta(\R)$ is the barycentre of $\eta$). Let $\sI_\eta$ be the smallest interval containing the support of $\eta$, and let $\{ \ell_\eta, r_\eta \}$ be the endpoints of $\sI_\eta$ (the support itself is the smallest closed set of full mass). If $\eta$ has an atom at $\ell_\eta$ then $\ell_\eta$ is included in $\sI_\eta$, and otherwise it is excluded, and similarly for $r_\eta$. For any Borel $B\subseteq\mathbb R^k$ and $\eta\in\sM^k$, we write $\eta\lvert_B$ for the restriction of $\eta$ to $B$.

For $\alpha \geq 0$ and $\beta \in \R$ let $\sD(\alpha, \beta)$ denote the set of non-decreasing, convex functions $f:\R \mapsto \R_+$ such that
\[ \lim_{ z \downarrow -\infty}  \{ f(z) \} =  0, \hspace{10mm} \lim_{z \uparrow \infty} \{ f(z) - (\alpha z- \beta) \}   =0. \]
Then, when $\alpha = 0$, $\sD(0,\beta)$ is empty unless $\beta = 0$ and then $\sD(0,0)$
contains one element, the zero function.

For $\eta\in\sM$, define the functions $P_\eta,C_\eta : \R \mapsto \R^+$ by
\begin{equation*}
P_\eta(k) := \int_{\R} (k-x)^+ \eta(dx),\quad k\in\R,
\end{equation*}
and
\begin{equation*}
C_\eta(k) := \int_{\R} (x-k)^+ \eta(dx),\quad k\in\R,
\end{equation*}
respectively. Then $P_\eta(k) \geq 0 \vee  (\eta(\R) k - \overline{\eta} )$ and $C_\eta(k) \geq 0 \vee (\overline{\eta} - \eta(\R)k)$. Also, the Put-Call parity holds: $C_\eta(k) - P_\eta(k) = (\overline{\eta}- \eta(\R)k)$.

The following properties of $P_\eta$ can be found in Chacon~\cite{chacon1977potential}, and Chacon and Walsh~\cite{chacon1976one}: $P_\eta \in \sD(\eta(\R), \overline{\eta})$ and $\{k : P_{\eta}(k) > (\eta(\R)k - \overline{\eta})^+  \} =  \{k : C_{\eta}(k) > (\overline{\eta}- \eta(\R)k)^+\} =(\ell_\eta,r_\eta)$. Conversely (see, for example, Proposition 2.1 in Hirsch et al. \cite{hirsch2012new}),  if $h$ is a non-negative, non-decreasing and convex function with $h \in \sD(k_m,k_f)$ for some numbers $k_m \geq 0$ and $k_f\in\R$ (with $k_f = 0$ if $k_m=0$), then there exists the unique measure $\eta\in\sM$, with total mass $\eta(\R)=k_m$ and first moment $\overline{\eta}=k_f$, such that $h=P_{\eta}$. In particular, $\eta$ is uniquely identified by the second derivative of $h$ in the sense of distributions. Furthermore, $P_\eta$ and $C_\eta$ are related to the potential $U_\eta$, defined by
\begin{equation*}
U_\eta(k) : =   \int_{\R} |k-x| \eta(dx),\quad k\in\R,
\end{equation*}
by $U_\eta=C_\eta+P_\eta$. We will call $P_\eta$ (and $C_\eta$) a modified potential. Finally note that all three second derivatives $C^{\prime\prime}_{\eta},P^{\prime\prime}_{\eta}$ and $U_\eta^{\prime\prime}/2$ identify the same underlying measure  $\eta$.

For $\eta,\chi\in\sM$, let
\begin{equation}\label{eq:sP}
\sC(\eta,\chi):=\{\tilde{P} \in \sD( \eta(\R), m)\ \textrm{for some }m\in\R:P_\chi-\tilde{P}\textrm{ is convex}\}.
\end{equation}

For $\eta,\chi\in\sM$, we write $\eta\leq\chi$ if $\eta(A) \leq \chi(A)$ for all Borel measurable subsets $A$ of $\R$, or equivalently if
\begin{equation*}
\int fd\eta\leq\int fd\chi,\quad \textrm{for all non-negative }f:\R\mapsto\R_+.
\end{equation*}
Since $\eta$ and $\chi$ can be identified as second derivatives of $P_\chi$ and $P_\eta$ respectively, we have $\eta\leq\chi$ if and only if $P_\chi-P_\eta$ is convex, i.e., $P_\eta$ has a smaller curvature than $P_\chi$. Note that there is one-to-one correspondence between measures $\theta\in\sM$ with $\theta(\R)=\eta(\R)$ and $\theta\leq\chi$, and functions $\tilde P\in\sP(\eta,\chi)$.

We further introduce some relevant stochastic orders. For $\eta,\chi\in\sM$ with $\eta(\R)\leq\chi(\R)$, we write $\eta\leq_c\chi,\eta\leq_{cd}\chi,\eta\leq_{ci}\chi,\eta\leq_{pc}\chi,\eta\leq_{pcd}\chi,\eta\leq_{pci}\chi$ if $\int_\R fd\eta\leq\int_\R fd\chi$ for all $f:\R\to\R$ that are convex, convex and non-increasing, convex and non-decreasing, non-negative and convex, non-negative, convex and non-increasing, and non-negative, convex and non-decreasing, respectively. (Note that if $\eta\leq_c\chi$, or $\eta\leq_{cd}\chi$, or $\eta\leq_{ci}\chi$, then automatically $\eta(\R)=\chi(\R)$.) If $\eta(\R)=\chi(\R)$ and $\int_\R fd\eta\leq\int_\R fd\chi$ for all non-decreasing $f:\R\to\R$, then we write $\eta\leq_{sto}\chi$.

For $\eta,\chi\in\sM$ with $\eta(\R)=\chi(\R)$, we write $\pi\in \Pi(\eta,\chi)$ if $\pi\in \sM^2$ and
$$
\pi(A\times\R)=\eta(A)\quad\textrm{and}\quad\pi(\R\times A)=\chi(A),\quad\textrm{for all Borel }A\subseteq\R.
$$
Note that the (scaled) product measure $(\eta\otimes\chi)/\eta(\R)$ is an element of $\Pi(\eta,\chi)$. We will often represent $\pi\in \Pi(\eta,\chi)$ via disintegration with respect to $\eta$: $\pi(dx,dy)=\pi_x(dy)\eta(dx)$, where $(\pi_x)_{x\in\R}$ is a ($\eta$-a.s. unique) family of probability measures satisfying $\pi_x\in\sP$ for $\eta$-a.e. $x\in\R$.

For $\eta,\chi\in\sM$ with $\eta(\R)=\chi(\R)$, we write $\pi\in\Pi_M(\eta,\chi)$ if $\pi\in\Pi(\eta,\chi)$ and $\int_\R y\pi_x(dy)=x$ for $\eta$-a.e. $x\in\R$, i.e., $\Pi_M(\eta,\chi)$ denotes the set of martingale couplings of $\eta$ and $\chi$. Due to Strassen \cite{strassen1965existence} we have that $\Pi_M(\eta,\chi)\neq\emptyset$ if and only if $\eta\leq_c\chi$. Similarly, we write $\pi\in\Pi_{Sup}(\eta,\chi)$ (resp. $\pi\in\Pi_{Sub}(\eta,\chi)$) if $\pi\in\Pi(\eta,\chi)$ and $\int_\R y\pi_x(dy)\leq x$ (resp. $\int_\R y\pi_x(dy)\geq x$) for $\eta$-a.e. $x\in\R$, i.e., $\Pi_{Sup}(\eta,\chi)$ (resp. $\Pi_{Sub}(\eta,\chi)$) denotes the set of supermartingale (resp. submartingale) couplings of $\eta$ and $\chi$. Then $\Pi_{Sup}(\eta,\chi)\neq\emptyset$ (resp. $\Pi_{Sub}(\eta,\chi)\neq\emptyset$) if and only if $\eta\leq_{cd}\chi$ (resp. $\eta\leq_{ci}\chi$).

\section{Weak Optimal Transport}\label{sec:wot}
Fix $\mu,\nu\in\sM$ with $\mu(\R)=\nu(\R)$, and a convex $h:\R\to\R_+$ with $h(0)=0$. The goal of the Weak Optimal Transport (WOT) theory is to determine the value
  \begin{equation}\label{eq:WOT_l1} 
   V_h(\mu,\nu):=\inf_{\pi\in\Pi(\mu,\nu)}\int_\R\left\lvert x-\int_\R y\pi_x(dy)\right\lvert \mu(dx)= \inf_{\pi\in\Pi(\mu,\nu)}\int_\R h\left( x-\overline{\pi_x}\right) \mu(dx),
  \end{equation}
  and the associated optimal couplings.

  \begin{remark} The assumption that $h$ is non-negative with $h(0)=0$ is not restrictive. Indeed, if $\tilde h:\R\to\R$ is convex, then, for any $a\in [\tilde h'(0-),\tilde h'(0+)]$, $x\mapsto \hat h(x)=[\tilde h(x)-\tilde h(0)-ax]$ is convex, $\hat h(0)=0\in[\hat h'(0-),\hat h'(0+)]$, and therefore $\hat h\geq 0$ on $\R$. In particular, $ V_{\hat h}(\mu,\nu)=V_{\tilde h}(\mu,\nu)-\tilde h(0)-a(\overline\mu-\overline\nu)$, and both problems share the same optimizers.
  \end{remark}
\subsection{WOT with $L^1$ costs}
In this section we consider a cost function $k\mapsto h(k)=\lvert k\lvert$, and thus in what follows write $V:=V_h$, where $V_h$ is defined in \eqref{eq:WOT_l1}.

From the previous works we already know that there exists an optimizer $\pi^*\in\Pi(\mu,\nu)$, so that
\begin{equation}\label{eq:WOT_l1+}
V(\mu,\nu)=\int_\R\left\lvert x-\int_\R y\pi^*_x(dy)\right\lvert \mu(dx),
\end{equation}
however, it is not unique. Hence, the goal is to compute the value \eqref{eq:WOT_l1+} and explicitly construct (at least one)
$$
\pi^*\in{\arg\min}_{\pi\in\Pi(\mu,\nu)}\int_\R\left\lvert x-\int_\R y\pi_x(dy)\right\lvert \mu(dx).
$$

First, observe that by Jensen's inequality we have that
\begin{equation}\label{eq:WOT_lb}
    \int_\R\left\lvert x-\int_\R y\pi_x(dy)\right\lvert \mu(dx)\geq \left\lvert \int_\R\left( x-\int_\R y\pi_x(dy)\right) \mu(dx) \right\lvert=\left\lvert \overline\mu-\overline\nu\right\lvert,\quad \pi\in\Pi(\mu,\nu).
\end{equation}
The natural question is then, for which $\mu,\nu\in\sM$ does there exist $\pi^*\in\Pi(\mu,\nu)$ such that
$$
\int_\R\left\lvert x-\int_\R y\pi^*_x(dy)\right\lvert \mu(dx)=\lvert \overline\mu-\overline\nu\lvert.
$$
Then any such coupling is automatically optimal for \eqref{eq:WOT_l1+}.

Suppose $\mu\leq_c\nu$. Then $\overline\mu=\overline\nu$. On the other hand, we also have that $\Pi_M(\mu,\nu)\neq\emptyset$, and thus there exists $\tilde\pi\in\Pi(\mu,\nu)$ such that $\int_\R y\tilde\pi_x(dy)=x$ for $\mu$-a.e. $x\in\R$. It follows that
$$
V(\mu,\nu)=\inf_{\pi\in\Pi_M(\mu,\nu)}\int_\R\left\lvert x-\int_\R y\tilde\pi_x(dy)\right\lvert \mu(dx)=0=\lvert \overline\mu-\overline\nu\lvert.
$$

As the next lemma shows, we can obtain a similar result in the cases when either $\mu\leq_{cd}\nu$ or $\mu\leq_{ci}\nu$. (Note that if $\mu\leq_c\nu$ then both $\mu\leq_{cd}\nu$ and $\mu\leq_{ci}\nu$.)

\begin{lemma}\label{lem:WOTsuper}
Suppose $\mu,\nu\in\sM$ with $\mu\leq_{cd}\nu$ (resp. $\mu\leq_{ci}\nu$). Then for any $\tilde\pi\in\Pi_{Sup}(\mu,\nu)$ (resp. $\hat\pi\in\Pi_{Sub}(\mu,\nu)$) we have that
\begin{align*}
V(\mu,\nu)&=\int_\R\left\lvert x-\int_\R y\tilde\pi_x(dy)\right\lvert \mu(dx)=\lvert \overline\mu-\overline\nu\lvert=(\overline\mu-\overline\nu)\\
\Big(resp.~V(\mu,\nu)&=\int_\R\left\lvert x-\int_\R y\hat\pi_x(dy)\right\lvert \mu(dx)=\lvert \overline\mu-\overline\nu\lvert=(\overline\nu-\overline\mu)\Big).
\end{align*}
\end{lemma}
\begin{proof}
    We only prove the case $\mu\leq_{cd}\nu$ (the case $\mu\leq_{ci}\nu$ follows by symmetry). If $\mu\leq_{cd}\nu$ then $\overline\nu\leq \overline\mu$ and $\Pi_{Sup}(\mu,\nu)\neq\emptyset$. Since, for any $\tilde\pi\in\Pi_{Sup}(\mu,\nu)\subseteq\Pi(\mu,\nu)$, we have that $\int_\R y\tilde\pi_x(dy)\leq x$ for $\mu$-a.e. $x$, it follows that
    $$
\int_\R\left\lvert x-\int_\R y\tilde\pi_x(dy)\right\lvert \mu(dx)=\int_\R\left( x-\int_\R y\tilde\pi_x(dy)\right) \mu(dx)=(\overline\mu-\overline\nu)=\lvert \overline\mu-\overline\nu\lvert.
    $$
    This, combined with \eqref{eq:WOT_lb}, proves the claim.
\end{proof}

In order to deal with the general case we introduce two constants:
\begin{equation}\label{eq:p_c}
p_{\mu,\nu}:=\sup_{k\in\R}\left\{P_\mu(k)-P_\nu(k)\right\}\quad\textrm{and}\quad c_{\mu,\nu}:=\sup_{k\in\R}\left\{C_\mu(k)-C_\nu(k)\right\}.
\end{equation}
Using that $\mu(\R)=\nu(\R)$ together with the Put-Call parity we then have that $$p_{\mu,\nu}=c_{\mu,\nu}+(\overline\nu-\overline\mu).$$

Now, note that $p_{\mu,\nu},c_{\mu,\nu}\in\R_+$. This can be easily seen by observing that $k\mapsto\left\{P_\mu(k)-P_\nu(k)\right\}$ (resp. $k\mapsto\left\{C_\mu(k)-C_\nu(k)\right\}$) is continuous, and converges to $0$ (resp. $(\overline\mu-\overline\nu)$) when $k\to-\infty$ and to $(\overline\nu-\overline\mu)$ (resp. 0) when $k\to\infty$.

Further observe that,
\begin{align*}
    p_{\mu,\nu}&=\inf\left\{m\geq0:P_\nu(k)+m\geq P_\mu(k),~k\in\R\right\},\\
    c_{\mu,\nu}&=\inf\left\{m\geq0:C_\nu(k)+m\geq C_\mu(k),~k\in\R\right\}.
\end{align*}
Indeed, if $p\in\R_+$ (resp. $c\in\R_+$) is another constant such that $P_\nu(k)+p\geq P_\mu(k)$ (resp. $C_\nu(k)+c\geq C_\mu(k)$) for all $k\in\R$, then $p\geq P_\mu(k)-P_\nu(k)$ (resp. $c\geq C_\mu(k)-C_\nu(k)$) for all $k\in\R$, and thus $p_{\mu,\nu}\leq p$ (resp. $c_{\mu,\nu}\leq c$). 

Moreover, if $p_{\mu,\nu}>0$ (resp. $c_{\mu,\nu}>0$), then there exists $k_p\in\R$ (resp. $k_c\in\R$) such that $P_\nu(k_p)+p_{\mu,\nu}=P_\mu(k_p)$ (resp. $C_\nu(c_p)+c_{\mu,\nu}=C_\mu(c_p)$); see Bayraktar et al. [Lemma  B.6 and its proof]\cite{BayDengNorgilas}.

On the other hand, if $p_{\mu,\nu}=0$ (resp. $c_{\mu,\nu}=0$), then $P_\nu\geq P_\mu$ (resp. $C_\nu\geq C_\mu$) on $\R$, which is equivalent to $\mu\leq _{cd}\nu$ (resp. $\mu\leq_{ci}\nu$), and thus Lemma \ref{lem:WOTsuper} applies. Hence, without loss of generality, in what follows we can assume that $p_{\mu,\nu}>0$ and $c_{\mu,\nu}>0$.

Consider $k\mapsto \tilde {P}_{\mu,\nu}(k)$ defined by
\begin{equation}\label{eq:tildeP_munu}
\tilde {P}_{\mu,\nu}(k):=P_\nu(k)+p_{\mu,\nu},\quad k\in\R.
\end{equation}
Note that $\tilde P_{\mu,\nu}\geq P_\mu$ on $\R$, and (since we assumed that $p_{\mu,\nu}>0$) $$\left\{k\in\R: \tilde{P}_{\mu,\nu}(k)=P_\mu(k)\right\}\neq\emptyset.$$

 Let
 \begin{equation}\label{eq:x-x+}
x^-_{\mu,\nu}:=\inf\left\{k\in\R: \tilde{P}(k)=P_\mu(k)\right\},\quad x^+_{\mu,\nu}:=\sup\left\{k\in\R: \tilde{P}(k)=P_\mu(k)\right\}.
 \end{equation}
 Now note that $x^-_{\mu,\nu},x^+_{\mu,\nu}\in\R$. Indeed, $\lim_{k\to-\infty}\left\{\tilde P_{\mu,\nu} (k)-P_\mu(k)\right\}=p_{\mu,\nu}>0$ and thus $-\infty<x^-_{\mu,\nu}\leq x^+_{\mu,\nu}$. Hence we only need to argue that $x^+_{\mu,\nu}<\infty$. This follows by observing that (by the Put-Call parity and the assumption that $c_{\mu,\nu}>0$)
 \begin{align*}
\lim_{k\to\infty}\left\{ \tilde{P}_{\mu,\nu}(k)-P_\mu(k)\right\}&=\lim_{k\to\infty}\left\{ {P}_\nu(k)-P_\mu(k)\right\}+p_{\mu,\nu}\\&=
\lim_{k\to\infty}\left\{ {C}_\nu(k)-C_\mu(k)\right\}-(\overline\nu-\overline\mu)+p_{\mu,\nu}=c_{\mu,\nu}>0.
 \end{align*}
 Furthermore, by the continuity of $P_\nu$ and $P_\mu$, we have that $\tilde P(k)=P_\mu(k)$ for $k\in\{x^-_{\mu,\nu},x^+_{\mu,\nu}\}$.

 We now split $\mu$ into three sub-probability measures $\eta^-_{\mu,\nu},\eta^0_{\mu,\nu},\eta^+_{\mu,\nu}\in\sM$ such that $\mu=\eta^-_{\mu,\nu}+\eta^0_{\mu,\nu}+\eta^+_{\mu,\nu}$:
 \begin{equation}\label{eq:mu_decomposition}
\eta^-_{\mu,\nu}:=\mu\lvert_{(-\infty,x^-_{\mu,\nu})},\quad\eta^0_{\mu,\nu}:=\mu\lvert_{[x^-_{\mu,\nu},x^+_{\mu,\nu}]},\quad\eta^+_{\mu,\nu}:=\mu\lvert_{(x^+_{\mu,\nu},\infty)}.
 \end{equation}
  Define
 \begin{equation}\label{eq:delta_slope}
\Delta^-_{\mu,\nu}:= P'_\mu(x^-_{\mu,\nu}-),\quad \Delta^+_{\mu,\nu}:= P'_\mu(x^+_{\mu,\nu}+),
 \end{equation}
 where $f'(k-)$ (resp. $f'(k+)$) denotes the left (resp. right) derivative, at $k\in\R$, of a convex function $x\mapsto f(x)$. Note that $\Delta^{\pm}_{\mu,\nu}\in\partial\tilde P_{\mu,\nu}(x^\pm)$, where $\partial f (x)$ denotes the sub-differential (provided it exists) at $x\in\R$, of $k\mapsto f(k)$. Then, by construction, $$\eta^-_{\mu,\nu}(\R)=\Delta^-_{\mu,\nu},\quad\eta^+_{\mu,\nu}(\R)=1-\Delta^+_{\mu,\nu},\quad\eta^0(\R)=\Delta^+_{\mu,\nu}-\Delta^-_{\mu,\nu}.$$

Now the goal is to also split $\nu$ into $\chi^-_{\mu,\nu},\chi^0_{\mu,\nu},\chi^+_{\mu,\nu}\in\sM$ with $\nu=\chi^-_{\mu,\nu}+\chi^0_{\mu,\nu}+\chi^+_{\mu,\nu}$, but with an additional requirement that $\eta^-_{\mu,\nu}\leq_{ci}\chi^-_{\mu,\nu}$, $\eta^0_{\mu,\nu}\leq_c\chi^0_{\mu,\nu}$ and $\eta^+_{\mu,\nu}\leq_{cd}\chi^+_{\mu,\nu}$. Set
\begin{align}
    \label{eq:nu_decompisition}
    \chi^-_{\mu,\nu}&:=\nu\lvert_{(-\infty,x^-_{\mu,\nu})}+(\Delta^-_{\mu,\nu}-\tilde{P}_{\mu,\nu}'(x^-_{\mu,\nu}-))\delta_{x^-_{\mu,\nu}},\nonumber\\
    \chi^+_{\mu,\nu}&:=\nu\lvert_{(x^+_{\mu,\nu},\infty)}+(\Delta^+_{\mu,\nu}-\tilde P_{\mu,\nu}'(x^+_{\mu,\nu}+))\delta_{x^+_{\mu,\nu}},\\
    \chi^0_{\mu,\nu}&:=\nu-\chi^-_{\mu,\nu}-\chi^+_{\mu,\nu}.\nonumber
\end{align}
\begin{lemma}\label{lem:mu_nu_decomposition}
Let $\eta^-_{\mu,\nu},~\eta^+_{\mu,\nu},~\eta^0_{\mu,\nu},~\chi^-_{\mu,\nu},~\chi^+_{\mu,\nu},~\chi^0_{\mu,\nu}\in\sM$ be as in \eqref{eq:mu_decomposition} and \eqref{eq:nu_decompisition}. Then
 $$
 \eta^-_{\mu,\nu}\leq_{ci}\chi^-_{\mu,\nu},\quad\eta^0_{\mu,\nu}\leq_c\chi^0_{\mu,\nu},\quad \eta^+_{\mu,\nu}\leq_{cd}\chi^+_{\mu,\nu}.
 $$
 \end{lemma}
 \begin{proof}
 In order to ease the notation we suppress the dependence  on $(\mu,\nu)$ and write $\eta^-=\eta^-_{\mu,\nu}$, and similarly for other quantities.
 
 First consider $\eta^-$ and $\chi^-$. By construction, $\chi^-(\R)=\eta^-(\R)$, and, since $\tilde P\geq P_\mu$ on $(-\infty,x^-]$ and $\tilde P(x^-)= P_\mu(x^-)$, we have that
 \begin{align*}
     C_{\eta^-}(k)&=\left(P_\mu(k)-\left\{P_\mu(x^-)+\Delta^-(k-x^-)\right\}\right)I_{\{k\leq x^-\}}\\
     &\leq \left(\tilde P(k)-\left\{P_\mu(x^-)+\Delta^-(k-x^-)\right\}\right)I_{\{k\leq x^-\}}\\
     &= \left(\tilde P(k)-\left\{\tilde P(x^-)+\Delta^-(k-x^-)\right\}\right)I_{\{k\leq x^-\}}=C_{\chi^-}(k),\quad k\in\R.
 \end{align*}It follows that $\eta^-\leq_{ci}\chi^-$.

By symmetry, $\eta^+\leq_{cd}\chi^+$. Indeed, by construction we have that $\chi^+(\R)=\eta^+(\R)$, and, since $\tilde P\geq P_\mu$ on $[x^+,\infty)$ and $\tilde P(x^+)= P_\mu(x^+)$, we have that 
 \begin{align*}
     P_{\eta^+}(k)&=\left(P_\mu(k)-\left\{P_\mu(x^+)+\Delta^+(k-x^+)\right\}\right)I_{\{k\geq x^+\}}\\
     &\leq \left(\tilde P(k)-\left\{P_\mu(x^+)+\Delta^+(k-x^+)\right\}\right)I_{\{k\geq x^+\}}\\
     &= \left(\tilde P(k)-\left\{\tilde P(x^+)+\Delta^+(k-x^+)\right\}\right)I_{\{k\geq x^+\}}=P_{\chi^+}(k),\quad k\in\R.
 \end{align*}

 Finally, consider $\eta^0$ and $\chi^0$. It is clear that $\chi^0(\R)=\eta^0(\R)$, and thus we are left to argue that $\eta^0\leq_c\chi^0$. For this it is enough to show that $P_{\eta^0}\leq P_{\chi^0}$ on $\R$, and $\lim_{k\to\infty}P_{\eta^0}(k)=\lim_{k\to\infty}P_{\chi^0}(k)$. We have that
 \begin{align*}
     P_{\eta^0}(k)&=\left(P_\mu(k) -\left\{P_\mu(x^-)+\Delta^-(k-x^-)\right\}\right)I_{k\in[x^-,x^+]} +\left(P_\mu(x^+)+\Delta^+(k-x^+)\right)I_{\{k>x^+\}}\\
     &\leq 
     \left(\tilde P(k) -\left\{P_\mu(x^-)+\Delta^-(k-x^-)\right\}\right)I_{k\in[x^-,x^+]} +\left(P_\mu(x^+)+\Delta^+(k-x^+)\right)I_{\{k>x^+\}}\\
     &=
     \left(\tilde P(k) -\left\{\tilde P(x^-)+\Delta^-(k-x^-)\right\}\right)I_{k\in[x^-,x^+]} +\left(\tilde P(x^+)+\Delta^+(k-x^+)\right)I_{\{k>x^+\}}\\
     &=P_{\chi^0}(k),\quad k\in\R,
 \end{align*}
 from which the desired asymptotic behaviour at $\infty$ also follows. We conclude that $\mu^0\leq_c\nu^0$. 
\end{proof}
We now define a set of candidate optimal couplings $\Pi^*(\mu,\nu)\subseteq\Pi(\mu,\nu)$:
\begin{equation}\label{eq:optimal_pi}
\begin{aligned}
\pi^*\in\Pi^*(\mu,\nu)\quad\textrm{if}\quad\pi^*\in\Pi(\mu,\nu)\quad\textrm{and}\quad\pi^*_{\mu,\nu}:=\pi^-_{\mu,\nu}+\pi^0_{\mu,\nu}+\pi^+_{\mu,\nu},\quad\textrm{where}
\\
\pi^-_{\mu,\nu}\in\Pi_{Sub}(\eta^-_{\mu,\nu},\chi^-_{\mu,\nu}),\quad\pi^0_{\mu,\nu}\in\Pi_{M}(\eta^0_{\mu,\nu},\chi^0_{\mu,\nu}),\quad\pi^+_{\mu,\nu}\in\Pi_{Sup}(\eta^+_{\mu,\nu},\chi^+_{\mu,\nu}).
\end{aligned}
\end{equation}
Note that, by construction and Lemma \ref{lem:mu_nu_decomposition}, $\Pi^*(\mu,\nu)$ is non-empty.

 We are now ready to state and prove the main result of this section.
 
 \begin{theorem}\label{thm:coupling} Suppose $\mu,\nu\in\sM$ with $\mu(\R)=\nu(\R)$ and $\min\{p_{\mu,\nu},c_{\mu,\nu}\}>0$. Then
 
 \begin{align*}
V(\mu,\nu)=\int_\R\left\lvert x-\int_\R y\pi^*_x(dy)\right\lvert\mu(dx)&=(\overline{\chi^-_{\mu,\nu}}-\overline{\eta^-_{\mu,\nu}})+(\overline{\eta^+_{\mu,\nu}}-\overline{\chi^+_{\mu,\nu}})\\&=V(\eta^-_{\mu,\nu},\chi^-_{\mu,\nu})+ V(\eta^+_{\mu,\nu},\chi^+_{\mu,\nu}),\quad\textrm{for all }\pi^*\in\Pi^*({\mu,\nu}).
 \end{align*}
Furthermore, any optimizer for $V(\mu,\nu)$ is of the form \eqref{eq:optimal_pi}.
 \end{theorem}
 \begin{proof}
 Similarly as in the proof of Lemma \ref{lem:mu_nu_decomposition}, we ease the notation by suppressing the dependence on $(\mu,\nu)$, and write $\eta^-=\eta^-_{\mu,\nu}$, etc.
 
 We first compute the total weak transport cost with respect to an arbitrary $\pi^*\in\Pi^*(\mu,\nu)$:
     \begin{align*}     &\int_\R\left\lvert x-\int_\R y\pi^*_x(dy)\right\lvert\mu(dx)\\&=\int_\R\left\lvert x-\int_\R y\pi^*_x(dy)\right\lvert\eta^-(dx)
     +\int_\R\left\lvert x-\int_\R y\pi^*_x(dy)\right\lvert\eta^0(dx)
     +\int_\R\left\lvert x-\int_\R y\pi^*_x(dy)\right\lvert\eta^+(dx)\\
     &=\int_\R\left\lvert x-\int_\R y\pi^-_x(dy)\right\lvert\eta^-(dx)
     +\int_\R\left\lvert x-\int_\R y\pi^0_x(dy)\right\lvert\eta^0(dx)
     +\int_\R\left\lvert x-\int_\R y\pi^+_x(dy)\right\lvert\eta^+(dx)\\
     &=\int_\R\left(\int_\R y\pi^-_x(dy)-x\right)\eta^-(dx)
     +\int_\R\left(x-\int_\R y\pi^+_x(dy)\right)\eta^+(dx)\\
     &=(\overline{\chi^-}-\overline{\eta^-})+(\overline{\eta^+}-\overline{\chi^+});
     \end{align*}
     the first equality uses that $\mu=\eta^-+\eta^0+\eta^+$ (see \eqref{eq:mu_decomposition}), while  for the remaining equalities we used the definition of $\pi^*$ (see \eqref{eq:optimal_pi}).

     We will now show that, for an arbitrary $\tilde\pi\in\Pi(\mu,\nu)$,
     $$
\int_\R\left\lvert x-\int_\R y\tilde\pi_x(dy)\right\lvert\mu(dx)\geq(\overline{\chi^-}-\overline{\eta^-})+(\overline{\eta^+}-\overline{\chi^+}),
     $$
     which will finish the proof of the first statement.

     Fix $\tilde\pi\in\Pi(\mu,\nu)$. Then
     \begin{align*}
\int_\R\left\lvert x-\int_\R y\tilde\pi_x(dy)\right\lvert\mu(dx)&\geq \int_\R\left\lvert x-\int_\R y\tilde\pi_x(dy)\right\lvert\eta^-(dx)
+\int_\R\left\lvert x-\int_\R y\tilde\pi_x(dy)\right\lvert\eta^+(dx)\\
&\geq \left\lvert \overline{\eta^-}-\int_\R\int_\R y\tilde\pi_x(dy)\eta^-(dx)\right\lvert
+\left\lvert \overline{\eta^+}-\int_\R\int_\R y\tilde\pi_x(dy)\eta^+(dx)\right\lvert,
     \end{align*}
     where the second inequality follows from Jensen's inequality.
     
     Now define $\tilde\pi^-(dx,dy):=\tilde\pi_x(dy)\eta^-(dx)$ and note that $\tilde\pi^-\in\Pi(\eta^-,\tilde \chi^-)$  for some $\tilde\chi^-\in\sM$ with $\tilde\chi^-\leq \nu$ and $\tilde\chi^-(\R)=\eta^-(\R)$. Similarly, $\tilde\pi^+(dx,dy):=\tilde\pi_x(dy)\eta^+(dx)$ is such that $\tilde\pi^+\in\Pi(\eta^+,\tilde\chi^+)$ for some $\tilde\chi^+\in\sM$ with $\tilde\chi+\leq \nu$ and $\tilde\chi^+(\R)=\eta^+(\R)$. It follows that
     $$
     \int_\R\left\lvert x-\int_\R y\tilde\pi_x(dy)\right\lvert\mu(dx)\geq\left\lvert 
     \overline{\eta^-}-\overline{\tilde\chi^-}\right\lvert +\left\lvert \overline{\eta^+}-\overline{\tilde\chi^+}\right\lvert.
     $$
     Therefore, in order to finish the proof, it is sufficient to show that
     $$
\left\lvert 
     \overline{\eta^-}-\overline{\tilde\chi^-}\right\lvert +\left\lvert \overline{\eta^+}-\overline{\tilde\chi^+}\right\lvert\geq (\overline{\chi^-}-\overline{\eta^-})+(\overline{\eta^+}-\overline{\chi^+}).
     $$
     We will achieve this by showing that
     \begin{equation}\label{eq:optimalityEQ}
         \left\lvert 
     \overline{\eta^-}-\overline{\tilde\chi^-}\right\lvert \geq (\overline{\chi^-}-\overline{\eta^-})\quad\textrm{and}\quad \left\lvert \overline{\eta^+}-\overline{\tilde\chi^+}\right\lvert\geq (\overline{\eta^+}-\overline{\chi^+}).
     \end{equation}
     
     First, note that since $\chi^-$ is a restriction of $\nu$ to $(-\infty,x^-)$ (together with an appropriate amount of mass at $x^-$), any other measure $\hat\chi^-\leq\nu$ with $\hat\chi^-(\R)=\chi^-(\R)$ has a larger mean. Indeed,
     $$
     {\chi^-}((-\infty,k])\geq \tilde\chi^-((-\infty,k]),\quad k\in\R,$$
     i.e., $\chi^-\leq_{sto}\tilde\chi^-$. Equivalently,
     $$
\int_\R f(k) \chi^-(dk)\leq \int_\R f(k)\tilde\chi^-(dk),\quad\textrm{for all non-decreasing }k\mapsto f(k),
     $$
     and by taking $k\mapsto f(k)=k$ we obtain that $\overline{\chi^-}\leq\overline{\tilde\chi^-}$.
     
     In fact, if $k\mapsto g(k)$ is convex and non-decreasing, using that $\eta^-\leq_{ci}\chi^-\leq_{sto}\tilde\chi^-$ we have that
     $$
\int_\R g(k) \eta^-(dk)\leq \int_\R g(k) \chi^-(dk)\leq\int_\R g(k) \tilde\chi^-(dk)
     $$
     (i.e., $\eta^-\leq_{ci}\chi^-\leq_{ci}\tilde\chi^-$),
     so that $\overline{\eta^-}\leq \overline{\chi^-}\leq \overline{\tilde\chi^-}$. It follows that
     
     $$
     \left\lvert \overline{\eta^-}-\overline{\tilde\chi^-}\right\lvert=\overline{\tilde\chi^-}-\overline{\eta^-}\geq \overline{\chi^-}-\overline{\eta^-}.
     $$

     We now use a symmetric argument for $\eta^+,\chi^+,\tilde\chi^+$. Observe that $$
     {\chi^+}((k,\infty))\geq \tilde\chi^+((k,\infty)),\quad k\in\R,$$
     and therefore 
     $\tilde\chi^+\leq_{sto}\chi^+$, or equivalently,
     $$
\int_\R f(k) \tilde\chi^+(dk)\leq \int_\R f(k)\chi^+(dk),\quad\textrm{for all non-decreasing }k\mapsto f(k).
     $$
     Then, if $k\mapsto g(k)$ is convex and non-increasing, using that $\eta^+\leq_{cd}\chi^+$, we have that
     $$
\int_\R g(k) \eta^+(dk)\leq \int_\R g(k) \chi^+(dk)\leq\int_\R g(k) \tilde\chi^+(dk)
     $$
     (i.e., $\eta^+\leq_{cd}\chi^+\leq_{cd}\tilde\chi^+$),
     so that $\overline{\eta^+}\geq \overline{\chi^+}\geq \overline{\tilde\chi^+}$. It follows that
     $$
     \left\lvert \overline{\eta^+}-\overline{\tilde\chi^+}\right\lvert=\overline{\eta^+}-\overline{\tilde\chi^+}\geq \overline{\eta^+}-\overline{\chi^+}.
     $$
     
     Combining both cases shows that 
     $$
     V(\mu,\nu)=\int_\R\left\lvert x-\int_\R y\pi^*_x(dy)\right\lvert\mu(dx)=(\overline{\chi^-_{\mu,\nu}}-\overline{\eta^-_{\mu,\nu}})+(\overline{\eta^+_{\mu,\nu}}-\overline{\chi^+_{\mu,\nu}}).$$
     Furthermore, recall that by Lemma \ref{lem:mu_nu_decomposition}, $ \eta^-_{\mu,\nu}\leq_{ci}\chi^-_{\mu,\nu}$ and $\eta^+_{\mu,\nu}\leq_{cd}\chi^+_{\mu,\nu}$. Then from Lemma \ref{lem:WOTsuper} we further have that $V(\eta^-_{\mu,\nu},\chi^-_{\mu,\nu})=(\overline{\chi^-_{\mu,\nu}}-\overline{\eta^-_{\mu,\nu}})$ and $V(\eta^+_{\mu,\nu},\chi^+_{\mu,\nu})=(\overline{\eta^+_{\mu,\nu}}-\overline{\chi^+_{\mu,\nu}})$.

     Finally, if $\tilde\pi$ is another optimizer of \eqref{eq:WOT_l1}, then \eqref{eq:optimalityEQ} must hold with equality, and thus $\overline{\tilde\chi^-}=\overline{\chi^-}$ and $\overline{\tilde\chi^+}=\overline{\chi^+}$. But then since $\chi^-\leq_{sto}\tilde\chi^-$ and $\tilde\chi^+\leq_{sto}\chi^+$ (and all measures are sub-measures of $\nu$), we must have that $\tilde\chi^\pm=\chi^\pm$. In addition, the second marginal of $\tilde\pi^0(dx,dy):=\tilde\pi_x(dy)\eta^0(dx)$ is then given by $(\nu-\tilde\chi^--\tilde\chi^+)=\chi^0$. This proves the second statement of the theorem and thus finishes the proof.
 \end{proof}
 \subsection{{Connection to Wasserstein projection of $\mu$ on $\{\hat\mu\in\sM:\hat\mu\leq_c\nu\}$}}\label{subsec:Lipschitz_map}
In this section we work with $\mu,\nu \in\sM$ with $\mu(\R)=\nu(\R)$ and a convex $h:\R\to\R_+$ satisfying $h(0)=0$. Define
  $$
 \bar V_h(\mu,\nu):=\inf_{\pi\in\Pi(\mu,\nu)}\int_{\R\times\R}h(x-y)\pi(dx,dy).
  $$

We follow Backhoff-Veraguas et al. \cite{backhoff2020weak} and call $S:\R\to\R$ \textit{admissible} if
\begin{enumerate}
  \item[i)] $S$ is non-decreasing,
  \item[ii)] $S$ is 1-Lipschitz,
  \item[iii)] the push-forward measure $S(\mu)$ satisfies $S(\mu)\leq_c \nu$.
  \end{enumerate}
  
Gozlan et al. \cite{gozlan2018characterization} (see also Alfonsi et al. \cite{alfonsi2017sampling} for an equivalent result in higher dimensions) showed that
  \begin{equation}\label{eq:weak_equivalence}
 V_h(\mu,\nu)=\inf_{\hat\mu\leq_c\nu} \bar V_h(\hat\mu,\mu)=\inf_{\textrm{admissible }S} \bar V_h(S(\mu),\mu).
  \end{equation}
In particular, there exists an admissible ($\mu$-a.s. unique) map $T^*$, such that 
$$
 V_h(\mu,\nu)=\bar V_h(T^*(\mu),\mu)\quad\textrm{for all convex }h\geq 0\textrm{ with }h(0)=0,
$$
i.e., the optimizer $T^*$ (and the induced push-forward measure $\mu^*:=T^*(\mu)$) does not depend on the choice of the convex cost function $h$. Furthermore, the map $T^*$ can be uniquelly identified by one of the following conditions (see Backhoff-Veraguas et al. \cite{backhoff2020weak}):
\begin{enumerate}
\item $S(\mu)\leq_c T^*(\mu)\leq_c\nu\quad\textrm{for all admissible }S$,
\item $T^*$ is the unique admissible map which has slope 1 on each interval $(T^*)^{-1}(I)$, where $I$ is irreducible\footnote{Two measures $\eta,\chi\in\sM$ with $\eta(\R)=\chi(\R)$ satisfies $\eta\leq_c\chi$ if and only if $\overline\eta=\overline\chi$ and $P_\eta\leq P_\chi$ on $\R$. By the continuity of the potential functions, the set $U:=\{k\in\R: P_\eta(k)< P_\chi(k)\}$ is open and thus can be represented as a countable union of disjoint open intervals, $U=\bigcup_n I_n$. These intervals $I_n$ are called irreducible w.r.t. $(\eta,\chi)$.}  w.r.t. $(T^*(\mu),\nu)$.
\end{enumerate}

Recall $\eta^-_{\mu,\nu},~\eta^+_{\mu,\nu},~\chi^-_{\mu,\nu},~\chi^+_{\mu,\nu}\in\sM$ given by \eqref{eq:mu_decomposition} and \eqref{eq:nu_decompisition}. Now let $T^-,T^+$ be the optimal maps for $V_h(\eta^-_{\mu,\nu},\chi^-_{\mu,\nu})$ and $V_h(\eta^+_{\mu,\nu},\chi^+_{\mu,\nu})$, respectively, i.e., for any convex $h:\R\to\R_+$ with $h(0)=0$ we have that
\begin{equation}\label{eq:T-T+}
V_h(\eta^-_{\mu,\nu},\chi^-_{\mu,\nu})=\bar V_h(T^-(\eta^-_{\mu,\nu}),\eta^-_{\mu,\nu})\quad \textrm{and}\quad V_h(\eta^+_{\mu,\nu},\chi^+_{\mu,\nu})=\bar V_h(T^+(\eta^+_{\mu,\nu}),\eta^+_{\mu,\nu}).
\end{equation}

\begin{proposition}\label{prop:optimalT}
Let $\mu,\nu\in \sM$ with $\mu(\R)=\nu(\R)$ and $\min\{p_{\mu,\nu},c_{\mu,\nu}\}>0$. Define $\hat T:\R\to\R$ by 
$$
\hat T(x)=\begin{cases}T^-(x),\quad &x<x^-_{\mu,\nu}\\
x,\quad & x^-_{\mu,\nu}\leq x\leq x^+_{\mu,\nu}\\
T^+(x),\quad & x^+_{\mu,\nu}<x.
\end{cases}
$$
Then $T^*=\hat T$.
\end{proposition}
\begin{corollary}\label{cor:monotonicityT}
Let $\mu,\nu\in \sM$ with $\mu(\R)=\nu(\R)$ and $\min\{p_{\mu,\nu},c_{\mu,\nu}\}>0$. Then the optimal map $T^*$ is such that $x\mapsto D(x):=(T^*(x)-x)$ is non-increasing, and
$$
D> 0\textrm{ on }(-\infty,x^-_{\mu,\nu}),\quad D=0\textrm{ on }[x^-_{\mu,\nu},x^+_{\mu,\nu}],\quad D<0\textrm{ on }(x^+_{\mu,\nu},\infty).
$$
\end{corollary}
\begin{proof}
Since $T^-$ (resp. $T^+$) is optimal for $V_h(\eta^-_{\mu,\nu},\chi^-_{\mu,\nu})$ (resp. $V_h(\eta^+_{\mu,\nu},\chi^+_{\mu,\nu})$) and $\eta^-_{\mu,\nu}\leq_{ci}\chi^-_{\mu,\nu}$ (resp. $\eta^+_{\mu,\nu}\leq_{cd}\chi^+_{\mu,\nu}$), by Gozlan and Juillet \cite[Theorem 3.2]{gozlan2020mixture}) we further have that $\eta^-_{\mu,\nu}\leq_{sto} T^-(\eta^-_{\mu,\nu})$ (resp. $T^+(\eta^+_{\mu,\nu})\leq_{sto}\eta^+_{\mu,\nu}$), and therefore $T^*(x)=T^-(x)\geq x$ for $\mu$-a.e. $x\leq x^-_{\mu,\nu}$ (resp. $T^*(x)=T^+(x)\leq x$ for $\mu$-a.e. $x\geq x^+_{\mu,\nu}$). Combining this with the representation of $T^*$ in terms of the quantile function of $T^*(\mu)$ (see Alfonsi et al. \cite[Theorem 2.6]{alfonsi2017sampling}), we immediately have that $D$ is non-increasing. Finally, if $D(x)=0$ for some $x<x^-_{\mu,\nu}$ (resp. $x>x^+_{\mu,\nu}$), then $D=0$ on $[x,x^-_{\mu,\nu}]$ (resp. $[x^+_{\mu,\nu},x]$), which contradicts the definition of $x^-_{\mu,\nu}$ (resp. $x^+_{\mu,\nu}$).
\end{proof}
\begin{remark}
The assumption $\min\{p_{\mu,\nu},c_{\mu,\nu}\}>0$ covers the most general case. If  $\mu\leq_{cd}\nu$ (resp.  $\mu\leq_{ci}\nu$) we have that $p_{\mu,\nu}=0<c_{\mu,\nu}$ (resp. $c_{\mu,\nu}=0<p_{\mu,\nu}$). In this case, set $k^-_{\mu,\nu}:=-\infty$ (resp. $k^+_{\mu,\nu}:=+\infty$). Then the statement of Proposition \ref{prop:optimalT} still holds (but with this modified value of $k^-_{\mu,\nu}$ (resp. $k^+_{\mu,\nu}$)). Note that, in the case $\mu\leq_c\nu$, we take $k^-_{\mu,\nu}=-\infty<+\infty=k^+_{\mu,\nu}$ and then $x\mapsto T^*(x)=x$ is optimal.\end{remark}

\begin{proof}[Proof of Proposition \ref{prop:optimalT}]
    By Lemma \ref{lem:mu_nu_decomposition}, $\eta^-_{\mu,\nu}\leq_{ci}\chi^-_{\mu,\nu}$ and $\eta^+_{\mu,\nu}\leq_{cd}\chi^+_{\mu,\nu}$, and therefore $\sup\{k\in\textrm{supp}(\eta^-_{\mu,\nu})\}\leq \sup\{k\in\textrm{supp}(\chi^-_{\mu,\nu})\}=x^-_{\mu,\nu}$ and $x^+_{\mu,\nu}=\inf\{k\in\textrm{supp}(\chi^+_{\mu,\nu})\}\leq \inf\{k\in\textrm{supp}(\eta^+_{\mu,\nu})\}$, respectively. It follows that $T^-(x)\leq x^-_{\mu,\nu}$ for all $x<x^-_{\mu,\nu}$ and $T^+(x)\geq x^+_{\mu,\nu}$ for all $x>x^+_{\mu,\nu}$. Then since $T^-$ (resp. $T^+$) is admissible w.r.t. $(\eta^-_{\mu,\nu},\chi^-_{\mu,\nu})$ (resp. $(\eta^+_{\mu,\nu},\chi^+_{\mu,\nu})$), from the definition of $\hat T$ we immediately have that $\hat T$ is admissible w.r.t. $(\mu,\nu)$. Indeed, the definition of $T^-$ as an optimal map for $V_h(\eta^-_{\mu,\nu},\chi^-_{\mu,\nu})$ implies that $T^-(\eta^-_{\mu,\nu})\leq_c \chi^-_{\mu,\nu}$ (by symmetry we then also conclude that $T^+(\eta^+_{\mu,\nu})\leq_c \chi^+_{\mu,\nu}$). Moreover, since $\hat T (x)=x$ for $x\in[x^-_{\mu,\nu},x^+_{\mu,\nu}]$, we have that $\hat T(\mu)\lvert_{[x^-_{\mu,\nu},x^+_{\mu,\nu}]}=\mu\lvert_{[x^-_{\mu,\nu},x^+_{\mu,\nu}]}=\eta^0_{\mu,\nu}\leq_c \chi^0_{\mu,\nu}$, where $\chi^0_{\mu,\nu}$ corresponds to the restriction of $\nu$ to $(x^-_{\mu,\nu},x^+_{\mu,\nu})$ together with appropriate amounts of mass at the boundaries (recall \eqref{eq:nu_decompisition}). Hence $\hat T(\mu)\leq_c\chi^-_{\mu,\nu}+\chi^0_{\mu,\nu}+\chi^+_{\mu,\nu}=\nu$.
    
    Furthermore, $P_{\hat T(\mu)}(k)=P_\nu(k)$ for $k\in\{x^-_{\mu,\nu},x^+_{\mu,\nu}\}$, and therefore for any irreducible interval $I_n$ w.r.t. $(\hat T(\mu),\nu)$ we have that both, $I_n$ and $\hat T^{-1}(I_n)$, belong to either $(-\infty,x^-_{\mu,\nu})$, or $[x^-_{\mu,\nu},x^+_{\mu,\nu}]$, or $(x^+_{\mu,\nu},\infty)$. But $\hat T=T^-$ (resp. $\hat T=T^+$) on $(-\infty,x^-_{\mu,\nu})$ (resp. $(x^+_{\mu,\nu},\infty)$), and thus $\hat T$ and $T^-$ (resp. $T^+$) share the same irreducible intervals on $(-\infty,x^-_{\mu,\nu})$ (resp. $(x^+_{\mu,\nu},\infty)$). It follows that $\hat T$ has slope 1 on $\hat T^{-1}(I_n)$ for every irreducible $I_n$ satisfying either $I_n\subseteq (-\infty,x^-_{\mu,\nu})$ or $I_n\subseteq (x^+_{\mu,\nu},\infty)$. On the other hand, if $I_n\subset [x^-_{\mu,\nu},x^+_{\mu,\nu}]$, then $\hat T^{-1}(I_n)=I_n$, and therefore (since $\hat T(x)=x$ for all $x\in[x^-_{\mu,\nu},x^+_{\mu,\nu}]$), $\hat T$ also has slope 1 for such $I_n$. We conclude that $\hat T$ has slope 1 on each $\hat T^{-1}(I_n)$. Since this last property is satisfied only by $T^*$, we conclude that $\hat T= T^*$.
\end{proof}

We can ow use the results of Proposition \ref{prop:optimalT} and give a version of Theorem \ref{thm:coupling} for general cost functions $h$.

\begin{corollary}\label{cor:valueGeneral}
Let $\mu,\nu\in \sM$ with $\mu(\R)=\nu(\R)$ and $\min\{p_{\mu,\nu},c_{\mu,\nu}\}>0$. Suppose $h:\R\to\R_+$ is convex and $h(0)=0$. Then 
\begin{align*}
V_h(\mu,\nu)&=V_h(\eta^-_{\mu,\nu},\chi^-_{\mu,\nu})+V_h(\eta^+_{\mu,\nu},\chi^+_{\mu,\nu})\\&=\inf_{\pi\in\Pi_{Sub}(\eta^-_{\mu,\nu},\chi^-_{\mu,\nu})}\int_\R h\left(x-\int_\R y\pi_x(dy)\right) \eta^-_{\mu,\nu}(dx)\\&\quad\quad+\inf_{\pi\in\Pi_{Sup}(\eta^+_{\mu,\nu},\chi^+_{\mu,\nu})}\int_\R h\left( x-\int_\R y\pi_x(dy)\right) \eta^+_{\mu,\nu}(dx).
\end{align*}
\end{corollary}
\begin{proof}
    By Proposition \ref{prop:optimalT}, and since $h(0)=0$, we have that
    \begin{align*}
V_h(\mu,\nu)=\bar V_h(T^*(\mu),\mu)&=\int_\R h(x-T^*(x))\mu(dx)\\
&=\int_{(-\infty,x^-_{\mu,\nu})}h(x-T^-(x))\eta^-_{\mu,\nu}+\int_{(x^+_{\mu,\nu},\infty)}h(x-T^+(x))\eta^+_{\mu,\nu}\\
&=\bar V_h(T^-(\eta^-_{\mu,\nu}),\chi^-_{\mu,\nu}) + \bar V_h(T^+(\eta^+_{\mu,\nu}),\chi^+_{\mu,\nu})\\
&=V_h(\eta^-_{\mu,\nu},\chi^-_{\mu,\nu})+V_h(\eta^+_{\mu,\nu},\chi^+_{\mu,\nu}).
    \end{align*}
    On the other hand, since $\eta^-_{\mu,\nu}\leq_{ci}\chi^-_{\mu,\nu}$, by Gozlan and Juillet \cite[Theorem 3.2]{gozlan2020mixture} we have that
    $$
    V_h(\eta^-_{\mu,\nu},\chi^-_{\mu,\nu})=\inf_{\pi\in\Pi_{Sub}(\eta^-_{\mu,\nu},\chi^-_{\mu,\nu})}\int_\R h\left(x-\int_\R y\pi_x(dy)\right) \eta^-_{\mu,\nu}(dx).
    $$
    Since $\eta^+_{\mu,\nu}\leq_{cd}\chi^+_{\mu,\nu}$, a symmetric argument shows that 
    $$
    V_h(\eta^+_{\mu,\nu},\chi^+_{\mu,\nu})=\inf_{\pi\in\Pi_{Sup}(\eta^+_{\mu,\nu},\chi^+_{\mu,\nu})}\int_\R h\left(x-\int_\R y\pi_x(dy)\right) \eta^+_{\mu,\nu}(dx),
    $$
    which finishes the proof.
\end{proof}

\section{Constrained optimal transport}\label{sec:cot}
For fixed marginals $\mu,\nu\in\sP$, recall the definition of $\Pi^*(\mu,\nu)\subseteq\Pi(\mu,\nu)$; see \eqref{eq:optimal_pi}. Note that, due to Theorem \ref{thm:coupling},
\begin{equation}\label{eq:cotCouplings}
\Pi^*(\mu,\nu)=\argmin_{\pi\in\Pi(\mu,\nu)}\int_\R\lvert x-\overline{\pi_x}\lvert\mu(dx).
\end{equation}
The main goal of this section is to identify some canonical elements of this set.

For a (measurable) cost function $c:\R\times\R\to\R$, in this section we consider the following variant of the optimal transport problem:
\begin{equation}\label{eq:cot}
\inf_{\pi\in\Pi^*(\mu,\nu)}\int_{\R\times\R}c(x,y)\pi(dx,dy). \end{equation}

We first observe that \eqref{eq:cot} generalizes the (sub/super-) martingale optimal transport problems introduced by Beiglb\"ock et al. \cite{BHLP13}, Beiglb\"ock and Juillet \cite{BeiglbockJuillet:16}, and Nutz and Stebegg \cite{NutzStebegg.18}. Indeed, due to Lemma \ref{lem:WOTsuper}, we have that 
\begin{align*}
\mu\leq_c\nu&\iff\Pi^*(\mu,\nu)=\Pi_M(\mu,\nu),\\
\mu\leq_{cd}\nu&\iff\Pi^*(\mu,\nu)=\Pi_{Sup}(\mu,\nu),\\
\mu\leq_{ci}\nu&\iff\Pi^*(\mu,\nu)=\Pi_{Sub}(\mu,\nu).
\end{align*}

\begin{lemma}\label{lem:Pi*properties}
Let $\mu,\nu\in\sP$. Then, $\Pi^*(\mu,\nu)$, as in \eqref{eq:cotCouplings}, 
    is convex and compact (wrt weak topology induced by the continuous and bounded functions).
\end{lemma}
Lemma \ref{lem:Pi*properties} ensures that, under mild regularity conditions for $c$, the problem \eqref{eq:cot} admits an optimizer. The proof of Corollary \ref{cor:cotMinimizer} is based on the standard (lower semi-continuity and compactness) arguments (see, for example, Beiglb\"ock et al. \cite[Theorem 1]{BHLP13}), and thus omitted.
\begin{corollary}\label{cor:cotMinimizer}
    Let $\mu,\nu\in\sP$. Suppose $c:\R\times\R\to\R$ is lower semi-continuous and $c(x,y)\geq - K(\lvert x\lvert+\lvert y\lvert)$ for all $x,y\in\R$, for some constant $K>0$.

    Then there exists $\pi^*\in\Pi^*(\mu,\nu)$ satisfying
    $$
\int_{\R\times \R}c(x,y)\pi^*(dx,dy)=\inf_{\pi\in\Pi^*(\mu,\nu)}\int_{\R\times \R}c(x,y)\pi(dx,dy).
    $$
\end{corollary}
\begin{proof}[Proof of Lemma \ref{lem:Pi*properties}]
    Let $\pi^1,\pi^2\in\Pi^*(\mu,\nu)$. Then, for $\alpha\in[0,1]$, we have that $\pi:=\alpha\pi^1+(1-\alpha)\pi^2$ is an element of $\Pi(\mu,\nu)$, $\pi_x=\alpha\pi^1_x+(1-\alpha)\pi^2_x$, and thus also $\overline{\pi_x}=\alpha\overline{\pi^1_x}+(1-\alpha)\overline{\pi^2_x}$ for $\mu$-a.e. $x\in\R$. Then using the convexity of $\vert\cdot\lvert$ it follows that
    $$
    \int_\R \lvert x-\overline{\pi_x}\lvert\mu(dx)\leq \alpha\int_\R \lvert x-\overline{\pi^1_x}\lvert\mu(dx)+(1-\alpha)\int_\R \lvert x-\overline{\pi^2_x}\lvert\mu(dx)= V(\mu,\nu)
    $$
    (recall \eqref{eq:WOT_l1}), and thus $\pi\in\Pi^*(\mu,\nu)$, which proves the convexity of $\Pi^*(\mu,\nu)$.

We now turn to compactness. Since $\Pi(\mu,\nu)$ is compact and $\Pi^*(\mu,\nu)\subseteq \Pi(\mu,\nu)$, it is enough to show that $\Pi^*(\mu,\nu)$ is closed. However, due to Corollary \ref{thm:coupling}, we have that any $\pi\in\Pi^*(\mu,\nu)$ is of the form \eqref{eq:optimal_pi}. Hence the claim follows if $\Pi_M(\eta,\chi),\Pi_{Sup}(\eta,\chi)$ and $\Pi_{Sub}(\eta,\chi)$ are all closed (provided they are non-empty) for some fixed $\eta,\chi\in\sP$.

The closedness of $\Pi_M(\eta,\chi)$ was proved by Beiglb\"ock et al. \cite{BHLP13} We now argue (by suitably modifying the arguments of Beiglb\"ock et al. \cite{BHLP13}) that $\Pi_{Sup}(\eta,\chi)$ (and thus, by symmetry, also $\Pi_{Sub}(\eta,\chi)$) is closed.

Fix $\eta\leq_{cd}\eta$ so that $\Pi_{Sup}(\eta,\chi)\neq\emptyset$. Now note that
$$
\pi\in\Pi_{Sup}(\eta,\chi) \iff \int_\R I_A(x)(y-x)\pi(dx,dy)\leq0\quad\forall \textrm{ Borel }A\subseteq \R.
$$
By standard approximation arguments we can replace $I_A$ by any non-negative, continuous and bounded $f:\R\to\R$. It follows that
$$
\Pi_{Sup}(\eta,\chi)=\bigcap_{f\in C_b(\R),~f\geq0} \mathcal{I}_f:= \bigcap_{f\in C_b(\R),~f\geq0}\left\{\pi\in\Pi(\mu,\nu):\int_\R f(x)(y-x)\pi(dx,dy)\leq 0\right\}.
$$
Since the sub-level sets of a continuous function are closed, using the continuity (wrt the weak convergence) of  $\pi\mapsto \int_\R f(x)(y-x)\pi(dx,dy)$ (see Beiglb\"ock et al. \cite[Lemma 2.2.]{BHLP13}), we conclude that each $\mathcal{I}_f$ (and thus also $\Pi_{Sup}(\eta,\chi)$) is closed.
\end{proof}
\begin{remark}\label{rem:generalC}
In fact one can prove Lemma \ref{lem:Pi*properties} in a more general setting. Fix $\mu,\nu\in\sP$, $\tilde h:\R\times\sP\to \R$, and consider
$$
\Pi^*_{\tilde h}(\mu,\nu):=\argmin_{\pi\in\Pi(\mu,\nu)}\int_\R \tilde h(x,\pi_x)\mu(dx).
$$

If $\tilde h$ is convex wrt the second variable, then using the same arguments as in Lemma \ref{lem:Pi*properties} we arrive to the convexity of  $\Pi^*_{\tilde h}(\mu,\nu)$. 

Furthermore, if in addition $\tilde h$ is lower semi-continuous (with respect to the product topology) and bounded from below (in particular, if it satisfies Condition (A) of Backhoff-Veraguas et al. \cite[Definition 2.7]{backhoff2019existence}), then $\pi\mapsto \int_\R \tilde h(x,\pi_x)\mu(dx)$ is lower semi-continuous (see Backhoff-Veraguas et al. \cite[Proposition 2.8]{backhoff2019existence}) ). Since the level sets of a lower semi-continuous function are closed, it follows that $\Pi^*_{\tilde h}(\mu,\nu)$ is a closed subset of $\Pi(\mu,\nu)$, and thus compact.

Note that the above applies to $(x,\eta)\mapsto \tilde h(x,\eta):=h(x-\overline\eta) $, where $h:\R\to\R_+$ is convex with $h(0)=0$ (as in section \ref{subsec:Lipschitz_map}).    
\end{remark}

Since each $\pi\in\Pi^*(\mu,\nu)$ is of the form \eqref{eq:optimal_pi}, it follows that
\begin{align}\label{eq:cotDecomposition}
    \inf_{\pi\in\Pi^*(\mu,\nu)}\int cd\pi=
    \inf_{\pi\in\Pi_{Sub}(\eta^-_{\mu,\nu},\chi^-_{\mu,\nu})}\int cd\pi+\inf_{\pi\in\Pi_{M}(\eta^0_{\mu,\nu},\chi^0_{\mu,\nu})}\int cd\pi+\inf_{\pi\in\Pi_{Sup}(\eta^+_{\mu,\nu},\chi^+_{\mu,\nu})}\int cd\pi.    
\end{align}
We further note that $(\eta^0_{\mu,\nu}+\eta^-_{\mu,\nu})\leq_{cd}(\chi^0_{\mu,\nu}+\chi^-_{\mu,\nu})$ and
\begin{equation}\label{eq:barrierSup}
\inf_{\pi\in\Pi_{Sup}(\eta^0_{\mu,\nu}+\eta^+_{\mu,\nu},\chi^0_{\mu,\nu}+\chi^+_{\mu,\nu})}\int cd\pi=\inf_{\pi\in\Pi_{M}(\eta^0_{\mu,\nu},\chi^0_{\mu,\nu})}\int cd\pi+\inf_{\pi\in\Pi_{Sup}(\eta^+_{\mu,\nu},\chi^+_{\mu,\nu})}\int cd\pi.
\end{equation}
Indeed, by construction we have that $x^+_{\mu,\nu}$ corresponds to the \textit{maximal barrier} for $(\eta^0_{\mu,\nu}+\eta^+_{\mu,\nu})\leq_{cd}(\chi^0_{\mu,\nu}+\chi^+_{\mu,\nu})$ in the sense of Nutz and Stebegg \cite[Proposition 3.2]{NutzStebegg.18}, and thus every $\pi\in\Pi_{Sup}(\eta^0_{\mu,\nu}+\eta^+_{\mu,\nu},\chi^0_{\mu,\nu}+\chi^+_{\mu,\nu})$ couples $\eta^0_{\mu,\nu}\leq_c\chi^0_{\mu,\nu}$ using a martingale (recall that the supports of both measures $\eta^0_{\mu,\nu}\leq_c\chi^0_{\mu,\nu}$ are contained in $[x^-_{\mu,\nu},x^+_{\mu,\nu}]$, while the supports of $\eta^+_{\mu,\nu}\leq_{cd}\chi^+_{\mu,\nu}$ are contained in $[x^+_{\mu,\nu},\infty)$). A symmetric argument shows that
\begin{equation}\label{eq:barrierSub}
\inf_{\pi\in\Pi_{Sub}(\eta^-_{\mu,\nu}+\eta^0_{\mu,\nu},\chi^-_{\mu,\nu}+\chi^0_{\mu,\nu})}\int cd\pi=\inf_{\pi\in\Pi_{Sub}(\eta^-_{\mu,\nu},\chi^-_{\mu,\nu})}\int cd\pi+\inf_{\pi\in\Pi_{M}(\eta^0_{\mu,\nu},\chi^0_{\mu,\nu})}\int cd\pi.
\end{equation}
In the next section we consider \eqref{eq:cotDecomposition} for some specific cost functions.
\subsection{Spence-Mirrlees costs}\label{sec:minimalVariance}
Nutz and Stebegg \cite{NutzStebegg.18} considered the supermartingale optimal tranport problems with cost functions $c$ that satisfy certain \textit{twist} conditions. A function $f:\mathbb{R}^2\to\R$ is called \textit{first-order Spence-Mirrlees} if
\begin{equation}\label{eq:superSM1}
f(x_2,\cdot)-f(x_1,\cdot)\textit{ is strictly increasing for all }x_1<x_2,
\end{equation}
and \textit{second-order Spence-Mirrlees} if
\begin{equation}\label{eq:superSM2}
f(x_2,\cdot)-f(x_1,\cdot)\textit{ is strictly convex for all }x_1<x_2.
\end{equation}
Then $f$ is the \textit{supermartingale Spence-Mirrlees} if $f$ is second-order Spence-Mirrlees and $(-f)$ is first-order Spence-Mirrlees.

Let $\Gamma\subset\R^2$ and consider $(x,y_1),(x,y_2),(x',y')\in\Gamma$ with $y_1<y_2$. Then $\Gamma$ is
\begin{itemize}
    \item[(i)] second-order left-monotone if $y'\notin(y_1,y_2)$ whenever $x<x'$,
    \item[(ii)] second-order right-monotone if $y'\notin(y_1,y_2)$ whenever $x>x'$.
\end{itemize}
Furthermore, let $(\Gamma,M)\subseteq\R^2\times\R$ and consider $(x_1,y_1),(x_2,y_2)$ with $x_1<x_2$. The pair $(\Gamma,M)$ is supermartingale (resp. submartingale)
\begin{itemize}
    \item[(i)] first-order left-monotone if $y_1\leq y_2$ whenever $x_2\notin M$ (resp. $x_1\notin M$),
    \item[(ii)] first-order right-monotone if $y_2\leq y_1$ whenever $x_1\notin M$ (resp. $x_2\notin M$).
\end{itemize}

Then we have the following
\begin{theorem}[Nutz and Stebegg \cite{NutzStebegg.18}]\label{thm:NS} Suppose $\eta,\chi\in\sP$ with $\eta\leq_{cd}\chi$. Then there exists (nondegenerate, in the sense of Nutz and Stebegg \cite[Definition 7.5]{NutzStebegg.18}) $(\Gamma^I,M^I)\subseteq \mathcal{B}(\R^2)\times\mathcal{B}(\R)$ that is supermartingale first-order right-monotone and second-order left-monotone (resp. $(\Gamma^D,M^D)\subseteq \mathcal{B}(\R^2)\times\mathcal{B}(\R)$ that is supermartingale first-order left-monotone and second-order right-monotone) and the unique coupling $\pi^I\in\Pi_{Sup}(\eta,\chi)$ (resp. $\pi^D\in\Pi_{Sup}(\eta,\chi)$), the so-called \textit{increasing} (resp. \textit{decreasing}) supermartingale coupling, such that
$$
\pi^I(\Gamma^I)=1\textrm{ and }\pi^I\lvert_{M^I\times\R}\textrm{ is a martingale}\quad \textrm{(resp. }\pi^I(\Gamma^I)=1\textrm{ and }\pi^I\lvert_{M^D\times\R}\textrm{ is a martingale).}
$$

Furthermore, $\pi^I$ (resp. $\pi^D$) is the unique optimizer of $\sup_{\pi\in\Pi_{Sup}(\eta,\chi)}\int fd\pi$ and $\inf_{\pi\in\Pi_{Sup}(\eta,\chi)}\int fd\pi$ if $f$ and $(-f)$ (resp. $(-f)$ and $f$) are supermartingale Spence-Mirrlees, respectively.    
\end{theorem}

Theorem \ref{thm:NS} completely solves \eqref{eq:barrierSup} for supermartingale Spence-Mirrlees cost functions. Hence it remains to solve $\inf_{\pi\in\Pi_{Sub}(\eta^-_{\mu,\nu},\chi^-_{\mu,\nu})}\int cd\pi$. However, the solution depends on $c$ (within the class of supermartingale Spence-Mirrlees cost functions.)

Using the symmetric arguments to those of Nutz and Stebegg \cite{NutzStebegg.18}, one can show that, for $\eta\leq_{ci}\chi$, there exists the unique $\tilde \pi^I\in\Pi_{Sub}(\eta,\chi)$ (resp. $\tilde \pi^D\in\Pi_{Sub}(\eta,\chi)$) that is supported by $(\tilde\Gamma ^I,\tilde M^I)$ (resp. $(\tilde\Gamma ^D,\tilde M^D)$), which is submartingale first-order left-monotone and second-order left-monotone (resp. submartingale first-order right-monotone and second-order right-monotone). $\pi^I$ and $\tilde \pi^I$ (resp. $\pi^D$ and $\tilde \pi^D$) are similar in the sense that both couplings mimic the left-curtain (resp. right-curtain) martingale coupling (see Beiglb\"ock ad Juillet \cite{BeiglbockJuillet:16}) on their respective \textit{martingale} points, $M^I$ and $\tilde M^I$ (resp. $M^D$ and $\tilde M^D$). However, the behaviour of $\pi^I$ and $\tilde \pi^I$ (resp. $\pi^D$ and $\tilde \pi^D$) differ on strict supermartingale and strict submartingale points respectively; there $\pi^I$ (resp. $\pi^D$) mimics the \textit{antitone} (resp. \textit{quantile}) coupling, and thus is supported on a decreasing (resp. increasing) map, while $\tilde\pi^I$ (resp. $\tilde\pi^D$) is supported on an increasing (resp. decreasing) map,  mimicking the behaviour of the quantile (resp. antitone) coupling.

The above differences of $\pi^I,\pi^D,\tilde\pi^I,\tilde\pi^D$ are also reflected in the optimal transport problems for which these couplings appear as optimizers.

\begin{lemma}\label{lem:example}
Let $\eta,\chi\in\sP$ with $\eta\leq_{ci}\chi$. If $c$ (resp. $-c$) is supermartingale Spence-Mirrlees, then neither the decreasing submartingale coupling $\tilde\pi^D\in\Pi_{Sub}(\eta,\chi)$, nor the increasing submartingale coupling $\tilde\pi^I\in\Pi_{Sub}(\eta,\chi)$, maximizes
$$
\Pi_{Sub}(\eta,\chi)\ni \pi\mapsto \int cd\pi.
$$
In particular, an optimizer is submartingale first-order right-monotone and second order left-monotone (resp. submartingale first-order left-monotone and second-order right-monotone).
\end{lemma}
Using Theorem \ref{thm:NS} and Lemma \ref{lem:example} we can characterize the optimzers of \eqref{eq:cotDecomposition} for the costs satisfying supermartingale Spence-Mirrlees conditions:
\begin{corollary}\label{cor:optimalSM}
 Let $\mu,\nu\in\sP$, and suppose that $\pi^*$ is an optimizer of \eqref{eq:cotDecomposition}. Then
 \begin{itemize}
     \item[(i)] If $c$ is supermartingale Spence-Mirrlees, then $\pi^*\lvert_{[x^-_{\mu,\nu},\infty)\times \R}$ is the (unique) increasing supermartingale coupling $\pi^I\in\Pi_{Sup}(\eta^0_{\mu,\nu}+\eta^+_{\mu,\nu},\chi^0_{\mu,\nu}+\chi^+_{\mu,\nu})$, while $\pi^*\lvert_{(-\infty,x^-_{\mu,\nu}]\times \R}$ is submartingale first-order right-monotone and second-order left-monotone.
     \item[(ii)] If $(-c)$ is supermartingale Spence-Mirrlees, then $\pi^*\lvert_{[x^-_{\mu,\nu},\infty)\times \R}$ is the (unique) decreasing supermartingale coupling $\pi^D\in\Pi_{Sup}(\eta^0_{\mu,\nu}+\eta^+_{\mu,\nu},\chi^0_{\mu,\nu}+\chi^+_{\mu,\nu})$, while $\pi^*\lvert_{(-\infty,x^-_{\mu,\nu}]\times \R}$ is submartingale first-order left-monotone and second-order right-monotone.
 \end{itemize}

\end{corollary}
\begin{proof}{Proof of Lemma \ref{lem:example}}
    Let $(I_k)_{k\geq -1}$ be the irreducible components of $\eta\leq_{ci}\chi$ in the sense of Nutz and Stebegg \cite[Proposition 3.4]{NutzStebegg.18} (but with necessary adjustments required for the submartingale setting). In particular, $I_0:=(-\infty,x^*)$ (where $x^*:=\inf\{k\in\R:C_\eta(k)=C_\chi(k)\})$ is the maximal barrier), $(I_k)_{k=1,...,N}$ (with $N\leq\infty$) are the open components (i.e., disjoint open intervals) of $(x^*,\infty)\cap\{C_\eta<C_\chi\}$, and $I_{-1}=\R\setminus\bigcup_{k\geq 0}I_k$. Let $\mu_k=\mu\lvert_{I_k}$ for $-1\leq k\leq N$. Then there exists the unique decomposition $\chi=\sum_{k\geq -1}\chi_k$, s.t. 
    $$
    \mu_{-1}=\nu_{-1},\quad \eta_0\leq_{ci}\chi_0,\quad \eta_k\leq_c\chi_k\textrm{ for all }k\geq 1.
    $$
    Let $(J_k)_{k\geq 0}$ be such that each $J_k$ is $I_k$, together with its endpoints in the case $\chi_k$ has atoms at these endpoints. Any $\pi\in\Pi_{Sub}(\eta,\chi)$ is a sum of couplings of $(\eta_k,\chi_k)$.

    Define $\Sigma:=(I_{-1}\times I_{-1})\bigcup_{k\geq 0}(I_k\times J_k)$. Let $(\Gamma,M)\subseteq \mathcal{B}(\R^2)\times\R$ be nondegenerate, where $\Gamma\subseteq\Sigma$, $M:=M_0\cup M_1$ with Borel sets $M_0\subseteq I_0$ and $M_1\cup_{k\neq 0} I_k$.

    For any finite measures $\pi,\pi'$ on $\R^2$ (with finite first moment) with the same first marginal $\pi^1$, we call $\pi'$ a $(M_0,M_1)$-competitor of $\pi$ if
    \begin{equation}\label{eq:competitor}
\overline{\pi'_x}\geq \overline{\pi_x}\textrm{ for $\pi^1$-a.e. $x\in M_0$,}\quad\textrm{and}\quad\overline{\pi'_x}=\overline{\pi_x}\textrm{ for $\pi^1$-a.e. $x\in M_1$.}
    \end{equation}
    
    Now suppose that, in the case $\pi$ is finitely supported on $\Gamma$, $\int fd\pi\geq \int fd\pi'$ for any $(M_0,M_1)$-competitor  $\pi'$ (of $\pi$) that is concentrated on $\Sigma$ (here $f:\R\to\R_+$ is any Borel function). This is precisely a submratingale version of Nutz and Stebegg \cite[Theorem 5.2]{}.

    Now consider $(x_1,y_1),(x_2,y_2)\in\Gamma$ with $x_1<x_2$, and define
    $$
    \pi:=\frac{1}{2}\left\{\delta_{(x_1,y_1)+\delta_{(x_2,y_2)}}\right\},\quad \pi':=\frac{1}{2}\left\{\delta_{(x_1,y_2)+\delta_{(x_2,y_1)}}\right\}.
    $$
    
    Suppose that $y_2<y_1$, so that
    $$
\overline{\pi_{x_1}}>\overline{\pi'_{x_1}}\quad\textrm{and}\quad \overline{\pi_{x_2}}<\overline{\pi'_{x_2}}.
    $$

    Suppose that $x_1\notin M$ (so that it does not affect the definition of a competitor, see \eqref{eq:competitor}). 

    Suppose that $x_2\in M_1$, so that $x_2\in I_k$ for some $k\neq 0$. Then $y_2\in J_k$. On the other hand, since $x_1\notin M$, we must have that $x_1\in I_0$ and thus also $y_1\in J_0$. But $J_k$ is located to the right of $J_0$. It follows that $y_1\leq y_2$.
    
    Now suppose that $x_2\notin M_1$. Then $x_2\in I_0$, and it follows that $\pi'$ is a $(M_0,M_1)$-competitor of $\pi$. Furthermore, $x_i\in I_0$ and thus $y_i\in J_0$, for $i=1,2$, and thus $\pi'$ is supported on $\Sigma$. It follows that 
    $$
0\leq 2\left (\int cd\pi- \int cd\pi'\right)=c(x_2,y_2)-c(x_1,y_2)-(c(x_2,y_1)-c(x_1,y_1)).
    $$
    But the right-hand side of the equality is strictly negative if $c$ first-order Spence-Mirrlees, and thus we again conlcude that $y_1\leq y_2$. 

    Combining both cases we conclude that $(\Gamma,M)$ is submartingale first-order left-monotone, whenever $c$ is first-order Spence-Mirrlees.
    
    Now suppose that $(-c)$ is first-order Spence-Mirrlees, $y_1<y_2$, and $x_2\notin M$. Then $x_1< x_2<x^*$ and thus both $x_1,x_2\in I_0$, and also both $y_1,y_2\in J_0$. Furthermore, since $y_1<y_2$, we have that $\overline{\pi_{x_1}}<\overline{\pi'_{x_1}}$, and thus $\pi'$ is again a $(M_0,M_1)$-competitor of $\pi$. It follows that
    $$
0\leq 2\left (\int cd\pi- \int cd\pi'\right)=c(x_1,y_1)-c(x_2,y_1)-(c(x_1,y_2)-c(x_2,y_2)).
    $$
    But the right-hand side is strictly negative since $(-c)$ is first order Spence-Mirrlees. We conclude that $(\Gamma,M)$ is submartingale first-order right-monotone, whenever $(-c)$ is first-order Spence-Mirrlees.

    Furthermore, we can follow verbatim the proof of Nutz and Stebegg \cite[Proposition 7.8 (iii)-(iv)]{NutzStebegg.18} and conclude that $(\Gamma,M)$ is second-order left-monotone (resp. right-monotone), whenever $c$ (resp. $(-c)$) is second-order Spence-Mirrlees.

    The above (together with a submartingale version of Nutz and Stebegg \cite[Theorem 5.2]{NutzStebegg.18}) proves a submartingale version of Nutz and Stebegg \cite[Proposition 7.8]{NutzStebegg.18}. Then following the lines of the first part of the proof Nutz and Stebegg \cite[Corollary 9.4]{NutzStebegg.18} we conclude that an optimizer of $\sup_{\pi\in\Pi_{Sub}(\eta,\chi)}\int cd\pi$ is submartingale first-order right-monotone and second-order left-monotone (resp. submartingale first-order left-monotone and second-order right-monotone), whenever $c$ (resp. $(-c)$) is supermartingale Spence-Mirrlees.
\end{proof}
\begin{remark}
    In general there are more than one submartingale coupling of $\mu\leq_{ci}\nu$, that is submartingale first-order right-monotone and second-order left-monotone, or submartingale first-order left-monotone and second-order right-monotone. See Nutz and Stebegg \cite[Section 10]{NutzStebegg.18} for an equivalent statement in the supermartingale setting. However, one can modify the increasing and decreasing submartingale couplings $\tilde\pi^I,\tilde\pi^D$ on their respective martingale points to obtain couplings with aforamentioned monotonicity properties.

    For example, let $\tilde M^D$ be the martingale points of $\tilde\pi^D$, and let $\nu^{\tilde M^D}$ be the second marginal of $\tilde\pi^D\lvert_{\tilde M^D\times \R}$. Then if $\tilde\pi^{lc,\tilde M^d}$ denote the left-curtain martingale coupling of $\mu\lvert_{\tilde M^D}\leq_c\nu^{\tilde M^D}$ (which is second-order left-monotone), it follows that $\pi:=\tilde\pi^D\lvert_{\R\setminus \tilde M^D}+\tilde\pi^{lc,\tilde M^d}$ is submartingale first-order right-monotone and second-order left-monotone.

    By modifying the increasing submartingale coupling $\tilde\pi^I$ on its martingale points we can similarly obtain a submartingale coupling that is submartingale first-order left-monotone and second-order right-monotone.

The question whether these couplings are optimizers of $\int cd\pi$ for some class of costs functions $c$ (and thus canonical), we leave for the future research.
\end{remark}

While we were not able to explicitly solve \eqref{eq:cotDecomposition} for supermartingale Spence-Mirrlees costs, we can give a positive answer in the case we relax the second-order Spence-Mirrlees condition. This leads us to the cost function $(x,y)\mapsto c(x,y)=\pm(x-y)^2$.

For marginals $\mu,\nu\in\sP$, we have that
$$
\int_{\R\times\R}(x-y)^2\pi(dx,dy)=\int_\R x^2\mu(dx)+\int_{\R\times\R}-xy\pi(dx,dy)+\int_\R y^2\nu(dy),\quad\textrm{for all }\pi\in\Pi(\mu,\nu),
$$
and therefore
$$
\argmin_{\pi\in\Pi^*(\mu,\nu)}\int_{\R\times\R}(x-y)^2\pi(dx,dy)=\argmax_{\pi\in\Pi^*(\mu,\nu)}\int_{\R\times\R}xy\pi(dx,dy),$$
$$\argmax_{\pi\in\Pi^*(\mu,\nu)}\int_{\R\times\R}(x-y)^2\pi(dx,dy)=\argmin_{\pi\in\Pi^*(\mu,\nu)}\int_{\R\times\R}xy\pi(dx,dy).
$$
The above corresponds to the set of couplings in $\Pi^*(\mu,\nu)$ that produce the largest and smallest correlation of random variables $X\sim\mu$ and $Y\sim\nu$.
\begin{proposition}\label{prop:cotCOV}
Let $\mu,\nu\in\sP$. Let $\tilde\pi^D,\tilde\pi^I\in\Pi_{Sub}(\eta^-_{\mu,\nu},\chi^-_{\mu,\nu})$ be the decraesing and increasing submartingale couplings, respectively. Let $\pi^D,\pi^I\in\Pi_{Sup}(\eta^+_{\mu,\nu},\chi^+_{\mu,\nu})$ be the decraesing and increasing supermartingale couplings, respectively. Let $\pi^M\in\Pi_M(\eta^0_{\mu,\nu},\chi^0_{\mu,\nu})$ be an arbitrary martingale coupling.

Define
$$
\pi^{min}:=\tilde\pi^D+\pi^M+\pi^I\quad\textrm{and}\quad\pi^{max}:=\tilde\pi^I+\pi^M+\pi^D.
$$
Then $\pi^{min}$ and $\pi^{max}$ produce (among the elements of $\Pi^*(\mu,\nu)$) the smallest and largest covariance of $X\sim\mu$ and $Y\sim\nu$, respectively.
\end{proposition}
\begin{proof}
Since, for any $\eta,\chi\in\sM$ with $\eta\leq_c\chi$,
$$
\int_{\R\times\R}xy\pi(dx,dy)=\int_\R x^2\eta(dx),\quad\textrm{for all }\pi\in\Pi_M(\eta,\chi),
$$
we have that
$$
\argmax_{\pi\in\Pi_M(\eta^0_{\mu,\nu},\chi^0_{\mu,\nu})}\int_{\R^2}xy\pi(dx,dy)=\argmin_{\pi\in\Pi_M(\eta^0_{\mu,\nu},\chi^0_{\mu,\nu})}\int_{\R^2}xy\pi(dx,dy)=\Pi_M(\eta^0_{\mu,\nu},\chi^0_{\mu,\nu}).
$$
Therefore, solving \eqref{eq:cotDecomposition} with $c(x,y)=\pm xy$ is equivalent to (separately) solving
$$
\sup_{\pi\in\Pi_{Sub}(\eta^-_{\mu,\nu},\chi^-_{\mu,\nu})}\int_{\R\times\R}c(x,y)\pi(dx,dy)\quad\textrm{and}\quad\sup_{\pi\in\Pi_{Sup}(\eta^+_{\mu,\nu},\chi^+_{\mu,\nu})}\int_{\R\times\R}c(x,y)\pi(dx,dy).
$$

Let $c(x,y)=-xy$. Then, for $x_1<x_2$, we have that 
$$
y\mapsto -c(x_2,y)-(-c(x_1,y))=(x_2-x_1)y
$$
is strictly increasing (and thus $(-c)$ is first-order Spence-Mirrlees, see \eqref{eq:superSM1}) and linear (and thus both $c$ and $(-c)$ fail to be second-order Spence-Mirrlees, see \eqref{eq:superSM2}; but they are still convex). In particular $c$ is \textit{relaxed} supermartingale Spence-Mirrlees (in the sense of Nutz and Stebegg \cite[Definition 9.1]{NutzStebegg.18}), and thus by Nutz and Stebegg \cite[Corollary 9.4]{NutzStebegg.18} we have that 
$$
-\inf_{\pi\in\Pi_{Sup}(\eta^+_{\mu,\nu},\chi^+_{\mu,\nu})}\int_{\R^2}xy\pi(dx,dy)=\sup_{\pi\in\Pi_{Sup}(\eta^+_{\mu,\nu},\chi^+_{\mu,\nu})}\int_{\R^2}-xy\pi(dx,dy)=-\int_{\R^2}xy\pi^I(dx,dy).
$$
On the other hand, by Lemma \ref{lem:example} and its proof, we have that any maximizer of
$$
\sup_{\pi\in\Pi_{Sub}(\eta^-_{\mu,\nu},\chi^-_{\mu,\nu})}\int_{\R^2}-xy\pi(dx,dy)
$$
is submartingale first-order right-monotone. Since $\tilde \pi^D\in\Pi_{Sub}(\eta^-_{\mu,\nu},\chi^-_{\mu,\nu})$ is submartingale first-order right-monotone, the submartingale version of Nutz and Stebegg \cite[Corollary 9.4]{NutzStebegg.18} then gives that
$$
-\inf_{\pi\in\Pi_{Sub}(\eta^-_{\mu,\nu},\chi^-_{\mu,\nu})}\int_{\R^2}xy\pi(dx,dy)=\sup_{\pi\in\Pi_{Sub}(\eta^-_{\mu,\nu},\chi^-_{\mu,\nu})}\int_{\R^2}-xy\pi(dx,dy)=-\int_{\R^2}xy\tilde\pi^D(dx,dy).
$$

Combining all three optimization problems we conclude that
$$
\inf_{\pi\in\Pi^*(\mu,\nu)}\int_{\R^2}xy\pi(dx,dy)=\int_{\R^2}xy\pi^{min}(dx,dy).
$$
Using symmetric arguments we obtain that
$$
\sup_{\pi\in\Pi^*(\mu,\nu)}\int_{\R^2}xy\pi(dx,dy)=\int_{\R^2}xy\pi^{max}(dx,dy).$$ 
\end{proof}

\begin{remark}\label{rem:genralCosts}
One could investigate \eqref{eq:cotDecomposition} beyond Spence-Mirrlees costs functions. For example, consider $c(x,y)=\pm\lvert x-y \lvert$. In the martingale setting when $\mu\leq_c\nu$,
$$
\sup_{\pi\in\Pi_M(\mu,\nu)}\int cd\pi
$$
was studied by Hobson and Neuberger \cite{HobsonNeuberger}, and Hobson and Klimmek \cite{hobson2015robust}. However, the supermartingale (and thus also submartingale) version of their results are not yet present in the literature.
\end{remark}
\section{Generalized shadow measure}\label{sec:generalShadow}
Fix $\mu,\nu\in\sM$ with $\mu(\R)\leq \nu(\R)$, and define 
$$
\mathcal T_{\mu,\nu}:=\{\theta\in\sM:\theta(\R)=\mu(\R),\ \theta\leq\nu\}.
$$
Then $\mathcal T_{\mu,\nu}$ is the set of all target measures for $\mu$ in $\nu$. It is easy to see that $\mathcal T_{\mu,\nu}\neq\emptyset$. Indeed, $(\mu(\R)/\nu(\R))\nu\in \mathcal T_{\mu,\nu}$. Furthermore, if $\mu=(\mu_1+\mu_2)$ for some $\mu_1,\mu_2\in\sM$, then note that, for any $\theta\in \mathcal T_{\mu_1,\nu}$, we have that $\mathcal T_{\mu_2,\nu-\theta}\neq\emptyset$.

In the cases when $\mu\leq_{pc}\nu$ and $\mu\leq_{pcd}\nu$, Beiglb\"ock et al. \cite{beiglbock2022potential} and Bayraktar et al. \cite{BayDengNorgilas}, respectively, investigated the so-called \textit{shadow} measure that arises as a second (distributional) derivative of a convex function $k\mapsto P_\nu(k)-(P_\nu-P_\mu)^c(k)$, where $f^c$ denotes the largest convex minorant (i.e., convex hull) of $f:\R\to\R$. Here we study this measure for arbitrary $\mu$ and $\nu$, and establish an interesting connection to the optimal transport with weak $L^1$ costs (recall Section \ref{sec:wot}).

Define 
$$
p_{\mu,\nu}=\sup_{k\in\R}\{P_\mu(k)-P_\nu(k)\}\quad\textrm{and}\quad c_{\mu,\nu}=\sup_{k\in\R}\{C_\mu(k)-C_\nu(k)\}.
$$
(This generalizes the definition of $p_{\mu,\nu},c_{\mu,\nu}$, given in \eqref{eq:p_c} for probability measures.) Similarly as for \eqref{eq:p_c}, note that $k\mapsto\{P_\mu(k)-P_\nu(k)\}$ is continuous, $\lim_{k\to-\infty}\{P_\mu(k)-P_\nu(k)\}=0$ and $\lim_{k\to\infty}\{P_\mu(k)-P_\nu(k)\}=\lim_{k\to\infty}\left\{(\overline\nu-\overline\mu)-k\left(\nu(\R)-\mu(\R)\right)\right\}$, from which we conclude that $p_{\mu,\nu}\in[0,\infty)$. In particular, if $p_{\mu,\nu}>0$, then $p_{\mu,\nu}=(P_\mu(k^*)-P_\nu(k^*))$ for some $k^*\in\R$. A similar argument shows that $c_{\mu,\nu}\in[0,\infty)$, and, in the case $c_{\mu,\nu}>0$, $c_{\mu,\nu}=(C_\mu(k^*)-C_\nu(k^*))$ for some $k^*\in\R$. On the other hand, if $p_{\mu,\nu}=0$ (resp. $c_{\mu,\nu}=0$) then $\mu\leq_{pcd}\nu$ (resp. $\mu\leq_{pci}\nu$) and we can embed $\mu$ into $\nu$ using a supermartingale (resp. submartingale).
\begin{lemma}\label{lem:convex_hull_measure}
Let $\mu,\nu\in\sM$ with $\mu(\R)\leq\nu(\R)$. The second (distributional) derivative of $k\mapsto (P_\nu-P_\mu)^c(k)$ corresponds to a (unique) measure $\theta_{\mu,\nu}\in\sM$ such that $P_{\theta_{\mu,\nu}}(k)=(P_\nu-P_\mu)^c(k)+p_{\mu,\nu}$, $k\in\R$. Furthermore, $\theta_{\mu,\nu}(\R)=(\nu(\R)-\mu(\R))$ and $\overline{\theta_{\mu,\nu}}=\overline\nu-\overline\mu+c_{\mu,\nu}-p_{\mu,\nu}$.
\end{lemma}
\begin{proof}
Define $k\mapsto h(k):=(P_\nu-P_\mu)^c(k)+p_{\mu,\nu}$. In order to prove the claim we need to show that $h\in\sD(\nu(\R)-\mu(\R),\overline\nu-\overline\mu+c_{\mu,\nu}-p_{\mu,\nu})$. (Uniqueness is immediate since the measures are uniquely identified by their potential functions.) Convexity of $h$ is clear, and thus we are left to show that $h$ is non-decreasing and has a correct asymptotic behaviour.

Observe that, by the definitions of $p_{\mu,\nu}$ and $c_{\mu,\nu}$, and the Put-Call parity, we have that $(P_\nu(k)-P_\mu(k))\geq \max\{-p_{\mu,\nu},(\nu(\R)-\mu(\R))k-(\overline\nu-\overline\mu+c_{\mu,\nu})\}$ for all $k\in\R$. Therefore, $(P_\nu-P_\mu)^c(k)\geq \max\{-p_{\mu,\nu},(\nu(\R)-\mu(\R))k-(\overline\nu-\overline\mu+c_{\mu,\nu})\}$ for all $k\in\R$. It follows that
$$
P_\nu(k)-P_\mu(k)+p_{\mu,\nu}\geq h(k)\geq \max\{0,(\nu(\R)-\mu(\R))k-(\overline\nu-\overline\mu+c_{\mu,\nu}-p_{\mu,\nu})\}\geq 0,\quad k\in\R,
$$
which we will now use to deduce the desired properties of $h$.

First, we have that
\begin{equation}\label{eq:h}
0\leq \lim_{k\to-\infty} h(k)\leq p_{\mu,\nu},
\end{equation}
and thus, since $h$ is convex, $h$ must also be non-decreasing. This also shows that, in the case $p_{\mu,\nu}=0$, $\lim_{k\to-\infty} h(k)=0$. We now argue that $\lim_{k\to-\infty} h(k)=0$ holds when $p_{\mu,\nu}>0$ as well. Indeed, in this case, for some $k_1\in\R$ we have that
$$
0\leq h(k_1) \leq P_\nu(k_1)-P_\mu(k_1)+p_{\mu,\nu} =0,
$$
which, combined with the convexity of $h$ and \eqref{eq:h}, shows that $\lim_{k\to-\infty} h(k)=0$. (Indeed, if $h'(k_--)<0$ for some $k_-\in(-\infty,k_1]$, then for a small enough $k< k_-$ we must have that $h(k)>p_{\mu,\nu}$, and thus also $\lim_{k\to-\infty} h(k)>p_{\mu,\nu}$, a contradiction. Hence $h(k)=0$ for all $k\leq k_1$, from which the asymptotic behaviour follows.)

Secondly, we also have that
\begin{align*}&\lim_{k\to\infty}\left\{(\nu(\R)-\mu(\R))k-(\overline\nu-\overline\mu+c_{\mu,\nu}-p_{\mu,\nu})\right\}\\ \leq&\lim_{k\to\infty} h(k)\leq  \lim_{k\to\infty}\{(\nu(\R)-\mu(\R))k-(\overline\nu-\overline\mu-p_{\mu,\nu})\},
\end{align*}
and thus in the case $c_{\mu,\nu}=0$ we immediately have that
$$
\lim_{k\to\infty}h(k)= \lim_{k\to\infty}\{(\nu(\R)-\mu(\R))k-(\overline\nu-\overline\mu+c_{\mu,\nu}-p_{\mu,\nu})\}.
$$
If $c_{\mu,\nu}>0$, then $C_\nu(k_2)-C_\mu(k_2)=-c_{\mu,\nu}$ for some $k_2\in\R$, which by the Put-Call parity is equivalent to
$$
P_\nu(k_2)-P_\mu(k_2)+p_{\mu,\nu}=(\nu(\R)-\mu(\R))k_2-(\overline\nu-\overline\mu+c_{\mu,\nu}-p_{\mu,\nu}).
$$
It follows that 
\begin{align*}
(\nu(\R)-\mu(\R))k_2-(\overline\nu-\overline\mu+c_{\mu,\nu}-p_{\mu,\nu})
&=P_\nu(k_2)-P_\mu(k_2)+p_{\mu,\nu}\\
&\geq 
h(k_2)\\&\geq \max\{0,(\nu(\R)-\mu(\R))k_2-(\overline\nu-\overline\mu+c_{\mu,\nu}-p_{\mu,\nu})\}\\&\geq(\nu(\R)-\mu(\R))k_2-(\overline\nu-\overline\mu+c_{\mu,\nu}-p_{\mu,\nu}),
\end{align*}
and therefore
$$
h(k_2)=(\nu(\R)-\mu(\R))k_2-(\overline\nu-\overline\mu+c_{\mu,\nu}-p_{\mu,\nu}).
$$
Then, (similarly as for the asymptotic behaviour at $-\infty$) if $h'(k_+)>(\nu(\R)-\mu(\R))$ for some $k_+\in[k_2,+\infty)$, then for a large enough $k> k_+$, $h(k)>(\nu(\R)-\mu(\R))k-(\overline\nu-\overline\mu+c_{\mu,\nu}-p_{\mu,\nu})$, and thus (by the convexity of $h$) also $\lim_{k\to\infty}h(k)>\lim_{k\to\infty}\{(\nu(\R)-\mu(\R))k-(\overline\nu-\overline\mu+c_{\mu,\nu}-p_{\mu,\nu})\}$, a contradiction. It follows that $h(k)=(\nu(\R)-\mu(\R))k-(\overline\nu-\overline\mu+c_{\mu,\nu}-p_{\mu,\nu})$ for all $k\geq k_2$, and thus $
\lim_{k\to\infty}h(k)= \lim_{k\to\infty}\{(\nu(\R)-\mu(\R))k-(\overline\nu-\overline\mu+c_{\mu,\nu}-p_{\mu,\nu})\}.
$

Combining the convexity, monotonicity and the asymptotic behaviour (at $-\infty$ and $\infty$) of $k\mapsto h(k)$, we conclude that $h\in\sD(\nu(\R)-\mu(\R),\overline\nu-\overline\mu+c_{\mu,\nu}-p_{\mu,\nu})$, which finishes the proof.
\end{proof}

We now introduce the generalized shadow measure, which we uniquely identify by the associated potential function of a special form.
\begin{lemma}[Generalized shadow measure]\label{lem:shadow_potential}
The second (distributional) derivative of $k\mapsto P_\nu(k)- (P_\nu-P_\mu)^c(k)$ corresponds to the unique measure $S^\nu(\mu)\in \mathcal T_{\mu,\nu}$ that satisfies $$P_{S^\nu(\mu)}(k)= P_\nu(k)- (P_\nu-P_\mu)^c(k)-p_{\mu,\nu},~k\in\R,\quad\textrm{and}\quad\overline{S^\nu(\mu)}=\overline\mu +p_{\mu,\nu}-c_{\mu,\nu}.$$
\end{lemma}
\begin{proof}
Define $k\mapsto h(k):=(P_\nu-P_\mu)^c+p_{\mu,\nu}$. We need to show that $(P_\nu-h)\in\sP(\mu,\nu)$ (see \eqref{eq:sP}) (and that the underlying measure has a correct mean), which is equivalent to $(P_\nu-h)\in \sD(\mu(\R),\overline\mu+p_{\mu,\nu}-c_{\mu,\nu})$ and $[P_\nu-(P_\nu-h)]=h$ being convex. The latter, however, immediately follows from the definition of $h$.

We now show that $(P_\nu-h)\in\sD(\mu(\R),\overline\mu+p_{\mu,\nu}-c_{\mu,\nu})$. By Lemma \ref{lem:convex_hull_measure} we have that $h\in\sD(\nu(\R)-\mu(\R),\overline\nu-\overline\mu+c_{\mu,\nu}-p_{\mu,\nu})$ and therefore $\lim_{\lvert k\lvert\to\infty}\{P_\nu(k)- h(k)\}=\lim_{\lvert k\lvert\to\infty}\max\{0,\mu(\R)k-(\overline\mu+p_{\mu,\nu}-c_{\mu,\nu})\}$. Hence the claim follows if $(P_\nu-h)$ is convex. But this is immediate from Beiglb\"ock et al. \cite[Lemma 2.3]{beiglbock2021shadow}.
\end{proof} 
For any $\theta\in \mathcal T_{\mu,\nu}$, define $$p_\theta:=\sup_{k\in\R}\{P_\mu(k)-P_\theta(k)\}.$$ Note that, $\lim_{k\to-\infty}\{P_\mu(k)-P_\theta(k)\}=0$ and (since $\theta(\R)=\mu(\R)$) $\lim_{k\to\infty}\{P_\mu(k)-P_\theta(k)\}=\overline\theta-\overline\mu$, and therefore $p_\theta\in\R_+$. Moreover,
$$
p_\theta=-\inf_{k\in\R}\{P_\theta(k)-P_\mu(k)\}=-(P_\theta-P_\mu)^c(k),\quad k\in\R
$$ and, in particular, $\inf\{m\geq0:P_{\theta}+m\geq P_\mu\textrm{ on }\R\}=p_\theta$. Using the Put-Call parity we further have that $p_\theta=\sup_{k\in\R}\{C_\mu(k)-C_\theta(k)\}-\overline\mu+\overline\theta$, and by defining
$$
c_\theta:=\sup_{k\in\R}\{C_\mu(k)-C_\theta(k)\}\in\R,
$$
we obtain that $\overline\theta=\overline\mu+p_\theta-c_\theta$.
\begin{lemma}\label{lem:shadow_mean}Fix $\mu,\nu\in\sM$ with $\mu(\R)\leq\nu(\R)$, and let $S^\nu(\mu)$ be the shadow measure as in Lemma \ref{lem:shadow_potential}. Then
	$$p_{S^\nu(\mu)}=p_{\mu,\nu}\quad\textrm{and}\quad c_{S^\nu(\mu)}=c_{\mu,\nu}.$$
\end{lemma}
\begin{proof}
By Lemma \ref{lem:shadow_potential} we have that $S^\nu(\mu)\in \mathcal T_{\mu,\nu}$, and therefore $\overline{S^\nu(\mu)}=\overline\mu+p_{S^\nu(\mu)}-c_{S^\nu(\mu)}$. On the other hand, by Lemma \ref{lem:shadow_potential} we also have that $\overline{S^\nu(\mu)}=\overline\mu +p_{\mu,\nu}-c_{\mu,\nu}$. Hence it is enough to prove that $p_{S^\nu(\mu)}=p_{\mu,\nu}$. We have that
	\begin{align*}
	p_{S^\nu(\mu)}&=-(P_{S^\nu(\mu)}-P_\mu)^c\\&=-((P_\nu-P_\mu)-(P_\nu-P_\mu)^c-p_{\mu,\nu})^c=-((P_\nu-P_\mu)^c-(P_\nu-P_\mu)^c-p_{\mu,\nu})^c=p_{\mu,\nu},
	\end{align*}
	where the third equality is a direct consequence of Beiglb\"ock et al. \cite[Lemma 2.4]{beiglbock2021shadow}. 
\end{proof}
The next lemma describes a special minimality (in terms of potential functions) property of the shadow measure.
\begin{lemma}\label{lem:shadow_minimality} Fix $\mu,\nu\in\sM$ with $\mu(\R)\leq \nu(\R)$, and let $S^\nu(\mu)$ be the shadow measure. Then, for all $\theta \in \mathcal T_{\mu,\nu}$,
	$$P_\mu(k)\leq P_{S^\nu(\mu)}(k)+p_{S^\nu(\mu)}\leq P_\theta(k)+p_\theta,\quad k\in\R,$$
	and $p_{S^\nu(\mu)}\leq p_\theta$.
\end{lemma}
\begin{proof}
	First, $P_\mu\leq P_\theta+p_\theta$ on $\R$, and therefore
	$$
	P_\nu-P_\theta-p_\theta\leq P_\nu-P_\mu\quad\textrm{on }\R.
	$$
	Since $\theta\in T_{\mu,\nu}$, we have that $(\nu-\theta)\in\sM$, and therefore $k\mapsto 	[P_\nu(k)-P_\theta(k)-p_\theta]=[P_{\nu-\theta}(k)-p_\theta]$ is convex. It follows that
	$$
	P_\nu-P_\theta-p_\theta\leq (P_\nu-P_\mu)^c\leq P_\nu-P_\mu\quad\textrm{on }\R,
	$$
	from which (using Lemmas \ref{lem:shadow_potential} and \ref{lem:shadow_mean}) we conclude that
	$$
P_\mu(k)\leq P_{S^\nu(\mu)}(k)+p_{{S^\nu(\mu)}}=P_\nu(k) - (P_\nu-P_\mu)^c(k)  -p_{\mu,\nu}+p_{{S^\nu(\mu)}}\leq P_\theta(k)+p_\theta\quad k\in\R.
	$$
	Letting $k\to-\infty$ we obtain that $p_{S^\nu(\mu)}\leq p_\theta$. 
\end{proof}
\begin{corollary}\label{cor:shadow_minimality} Fix $\mu,\nu\in\sM$ with $\mu(\R)\leq \nu(\R)$. Then for all $\theta \in \mathcal T_{\mu,\nu}$ we have that
	$$C_\mu(k)\leq C_{S^\nu(\mu)}(k)+c_{S^\nu(\mu)}\leq C_\theta(k)+c_\theta,\quad k\in\R,$$
	and $c_{S^\nu(\mu)}\leq c_\theta$.
\end{corollary}
\begin{proof}
Since, for all $\theta\in \mathcal T_{\mu,\nu}$, $\overline\theta=\overline\mu+p_\theta-c_\theta$, the claim immediately follows from Lemma \ref{lem:shadow_minimality} and the Put-Call parity.
	\end{proof}

While Lemma \ref{lem:shadow_minimality} characterizes the shadow measure in terms of its potential function, the next result provides a characterization in terms of minimal (weak) transportation cost.

Now recall the setting of Section \ref{sec:wot}. For each $\theta\in \mathcal T_{\mu,\nu}$ consider
$$
V(\mu,\theta)=\inf_{\pi\in\Pi(\mu,\theta)}\int_\R \left\lvert x-\int_\R y\pi_x(dy) \right\lvert\mu(dx),
$$
which is equivalent to \eqref{eq:WOT_l1} since $\mu(\R)=\theta(\R)$.

\begin{theorem}\label{thm:shadow_minimalityWOT}
Fix $\mu,\nu\in\sM$ with $\mu(\R)\leq\nu(\R)$. The shadow measure (of $\mu$ in $\nu$) $S^\nu(\mu)$, as in Lemma \ref{lem:shadow_potential}, is the unique element of $\mathcal T_{\mu,\nu}$  satisfying 
\begin{enumerate}
    \item $S^\nu(\mu)\in {\arg\min}_{\theta\in \mathcal T_{\mu,\nu}} V(\mu,\theta)$,
    \item $
S^\nu(\mu)\leq_c \theta^*,\quad\textrm{for all }\theta^*\in {\arg\min}_{\theta\in \mathcal T_{\mu,\nu}} V(\mu,\theta)
$.
\end{enumerate}
\end{theorem}
\begin{proof}
Without loss of generality we assume that $\min\{p_{S^\nu(\mu)},c_{S^\nu(\mu)}\}>0$. Note that, by Lemma \ref{lem:shadow_minimality} we have that $\min\{p_{\theta},c_{\theta}\}\geq\min\{p_{S^\nu(\mu)},c_{S^\nu(\mu)}\}>0$.

Fix $\theta\in\mathcal T_{\mu,\nu}$. Then by Theorem \ref{thm:coupling} we have that
$$
V(\mu,\theta)=V(\eta^-_{\mu,\theta},\chi^-_{\mu,\theta})+V(\eta^+_{\mu,\theta},\chi^+_{\mu,\theta})=(\overline{\chi^-_{\mu,\theta}}-\overline{\eta^-_{\mu,\theta}})+(\overline{\eta^+_{\mu,\theta}}-\overline{\chi^+_{\mu,\theta}}).
$$
Since $S^\nu(\mu)\in\mathcal{T}_{\mu,\nu}$, the goal is to show that
$$
(\overline{\chi^-_{\mu,\theta}}-\overline{\eta^-_{\mu,\theta}})\geq(\overline{\chi^-_{\mu,S^\nu(\mu)}}-\overline{\eta^-_{\mu,S^\nu(\mu)}})\quad\textrm{and} \quad(\overline{\eta^+_{\mu,\theta}}-\overline{\chi^+_{\mu,\theta}})\geq\quad(\overline{\eta^+_{\mu,S^\nu(\mu)}}-\overline{\chi^+_{\mu,S^\nu(\mu)}}).$$
We will only treat the first inequality, the second inequality can be obtained using symmetric arguments. 

By Lemma \ref{lem:shadow_minimality}, we have that $x^-_{\mu,S^\nu(\mu)}\leq x^-_{\mu,\theta}$. Then, let $S_\theta$ be a restriction of $S^\nu(\mu)$ to $(-\infty,x^-_{\mu,\theta})$ together with an appropriate amount of mass $\alpha\leq S^\nu(\mu)(\{x^-_{\mu,\theta}\})$ such that $S_\theta(\mathbb{R})=S_\theta((-\infty,x^-_{\mu,\theta}])=\chi^-_{\mu,\theta}(\R)$. Then $P_{S_\theta}=P_{S^\nu(\mu)}$ on $(-\infty,x^-_{\mu,\theta}]$, $P_{S_\theta}$ is linear on $(x^-_{\mu,\theta},\infty)$ with slope $\chi^-_{\mu,\theta}(\R)$, $P_{S_{\theta}}(x^-_{\mu,\theta})+p_{S^\nu(\mu)}=P_{\chi^-_{\mu,\theta}}(x^-_{\mu,\theta})+p_{\theta}=P_\mu(x^-_{\mu,\theta})$ and $P_{\chi^-_{\mu,\theta}}+p_\theta\geq P_{S_{\theta}}+p_{S^\nu(\mu)}\geq P_\mu$ on $(-\infty, x^-_{\mu,\theta})$. It follows that $\eta^-_{\mu,\theta}\leq_{ci}S_{\theta}\leq_{ci}\chi^-_{\mu,\theta}$ and therefore
\begin{align}\label{eq:submartingale}
(\overline{\chi^-_{\mu,\theta}}-\overline{\eta^-_{\mu,\theta}})&\geq(\overline{S_\theta}-\overline{\eta^-_{\mu,\theta}})\\
&=(\overline{S_\theta-\chi^-_{\mu,S^\nu(\mu)}}-\overline{\eta^-_{\mu,\theta}-\eta^-_{\mu,S^\nu(\mu)}})+(\overline{\chi^-_{\mu,S^\nu(\mu)}}-\overline{\eta^-_{\mu,S^\nu(\mu)}})\nonumber\\
&=(\overline{\chi^-_{\mu,S^\nu(\mu)}}-\overline{\eta^-_{\mu,S^\nu(\mu)}})\nonumber.
\end{align}
The last equality is obtained as follows. Recall that $x^-_{\mu,S^\nu(\mu)}\leq x^-_{\mu,\theta}$, $P_{S_\theta}=P_{S^\nu(\mu)}$ on $(-\infty,x^-_{\mu,\theta}]$, $P_{S^\nu(\mu)}(x)+p_{S^\nu(\mu)}=P_\mu(x)$ for $x\in\{x^-_{\mu,S^\nu(\mu)},x^-_{\mu,\theta}\}$, and $P_{S_\theta}+p_{S^\nu(\mu)}\geq P_\mu$ on $(-\infty,x^-_{\mu,\theta}]$. Therefore $(\eta^-_{\mu,\theta}-\eta^-_{\mu,S^\nu(\mu)})\leq_c(S_\theta-\chi^-_{\mu,S^\nu(\mu)})$, and thus these measures have equal means. We conclude that $S^\nu(\mu)\in {\arg\min}_{\theta\in \mathcal T_{\mu,\nu}} V(\mu,\theta)$.

Furthermore, if $\hat\theta\in\argmin_{\sT_{\mu,\nu}}V(\mu,\theta)$, then we must have that $S_{\hat\theta}\leq_{ci}\chi^-_{\mu,\hat\theta}$ and \eqref{eq:submartingale} holds with equality. It follows that $S_{\hat\theta}\leq_{c}\chi^-_{\mu,\hat\theta}$ and therefore $p_{\hat\theta}=p_{S^\nu(\mu)}$. By symmetry, $c_{\hat\theta}=c_{S^\nu(\mu)}$, and, by combining Lemma \ref{lem:shadow_minimality} and Corollary \ref{cor:shadow_minimality}, we conclude that $S^\nu(\mu)\leq_c\hat\theta$. This finishes the proof.

\end{proof}

We next present the associativity property of the shadow measure. It is the main ingredient in the construction of the lifted shadow couplings.
\begin{lemma}
	\label{lem:shadow_assoc}
	Let $\mu,\nu\in\sM$ with $\mu(\R)\leq\nu(\R)$, and suppose that $\mu=\mu_1+\mu_2$ for some $\mu_1,\mu_2\in\sM$. Then 
	\begin{equation*}
	S^\nu(\mu_1+\mu_2)=S^\nu(\mu_1)+S^{\nu-S^\nu(\mu_1)}(\mu_2).
	\end{equation*}
	Furthermore, $p_{S^\nu(\mu)}=p_{S^\nu(\mu_1)}+p_{S^{\nu-S^\nu(\mu_1)}(\mu_2)}$.
\end{lemma}

\begin{proof}
	Note that $\mu_2(\R)\leq(\nu-S^\nu(\mu_1))(\R)=\nu(\R)-\mu_1(\R)$, so that $\mathcal T_{\mu_2,\nu-S^\nu(\mu_1)}\neq\emptyset$, and therefore $S^{\nu-S^\nu(\mu_1)}(\mu_2)$ is well-defined.
	
By Lemma \ref{lem:shadow_potential} we have that, on $\R$,
	$$
	P_{S^\nu(\mu_1+\mu_2)}= P_\nu - \left((P_\nu-P_{\mu_1})-P_{\mu_2}\right)^c-p_{\mu_1+\mu_2,\nu}= P_\nu - \left((P_\nu-P_{\mu_1})^c-P_{\mu_2}\right)^c-p_{\mu_1+\mu_2,\nu},
	$$
	where the second equality follows from Beiglb\"ock et al. \cite[Lemma 2.4]{beiglbock2021shadow}. Now, using Lemma \ref{lem:shadow_potential} for the inner convex hull and taking the constant $p_{\mu_1,\nu}$ outside the outer convex hull we have that
	$$
	P_{S^\nu(\mu_1+\mu_2)}=P_\nu - \left((P_\nu-P_{S^\nu(\mu_1)})-P_{\mu_2}\right)^c+p_{\mu_1,\nu}-p_{\mu_1+\mu_2,\nu}\quad\textrm{on }\R.
	$$
	Since $S^\nu(\mu_1)\leq\nu$, $(\nu-S^\nu(\mu_1))\in\sM$, and therefore $(P_\nu-P_{S^\nu(\mu_1)})=P_{\nu-S^\nu(\mu_1)}$. It follows that
	\begin{align*}
	P_{S^\nu(\mu_1+\mu_2)}&=P_\nu - \left(P_{\nu-S^\nu(\mu_1)}-P_{\mu_2}\right)^c+p_{\mu_1,\nu}-p_{\mu_1+\mu_2,\nu}\\&=P_\nu -\left(P_{\nu-S^\nu(\mu_1)}-P_{S^{\nu-S^\nu(\mu_1)}(\mu_2)}\right)+p_{\mu_2,\nu-S^\nu(\mu_1)}+p_{\mu_1,\nu}-p_{\mu_1+\mu_2,\nu}\\
	&=P_{S^\nu(\mu_1)}+P_{S^{\nu-S^\nu(\mu_1)}(\mu_2)}+p_{\mu_2,\nu-S^\nu(\mu_1)}+p_{\mu_1,\nu}-p_{\mu_1+\mu_2,\nu}\quad\textrm{on }\R,
	\end{align*}	
where we used Lemma \ref{lem:shadow_potential} for the second equality. Taking second (distributional) derivatives proves the associativity property.

For the second assertion, note that, since $\lim_{k\to-\infty}P_\theta(k)=0$ for all $\theta\in\sM$, we must have that $p_{\mu_2,\nu-S^\nu(\mu_1)}+p_{\mu_1,\nu}=p_{\mu_1+\mu_2,\nu}$. Then Lemma \ref{lem:shadow_mean} completes the proof.
\end{proof}

We finish this paper with the result regarding the existence of shadow couplings.

In the rest of this section we work with $\mu,\nu\in\sP$ (or, more generally, with measures that have equal total mass). We first recall the notions of the lift of $\mu$ and the lifted couplings of $\mu$ and $\nu$, given in the introduction.

Let $\lambda=\Leb_{[0,1]}$ be the Lebesgue measure on $[0,1]$. A lift of $\mu$ is a measure $\hat\mu$ on $[0,1]\times\R$ with first and second marginals given by $\lambda$ and $\mu$, respectively (i.e., $\hat\mu\in\Pi(\lambda,\mu)$). For each $\hat\mu\in\Pi(\lambda,\nu)$, define a family of measures on $\R$, $(\hat\mu_{[0,u]})_{u\in[0,1]}$, by setting, for each $u\in[0,1]$,
$$\hat\mu_{[0,u]}(A)=\hat\mu([0,u]\times A),\quad\textrm{for all Borel}\quad A\subseteq\R.
$$
For a fixed lift $\hat\mu$, let $\hat\Pi(\hat\mu,\nu)$ be the set of probability measures on $[0,1]\times\R\times\R$, such that $\hat\pi(A\times B\times\R)=\hat\mu(A\times B)$ and $\hat\pi([0,1]\times\R\times B)=\nu(B)$, for all Borel $A\subseteq[0,1]$, $B\subseteq \R$.

\begin{theorem}\label{thm:shadowCouplings} 
    Let $\mu,\nu\in\sP$. For each lift $\hat\mu\in\Pi(\lambda,\nu)$, there exists the unique $\hat\pi\in\hat\Pi(\hat\mu,\nu)$, such that
    $$
\int^u_0 d\hat\pi\in\Pi(\hat\mu_{[0,u]},S^\nu(\hat\mu_{[0,u]})),\quad\textrm{for all }u\in[0,1].
    $$
    In particular, $\int^1_0d\hat\pi\in\Pi^*(\mu,\nu)$.
\end{theorem}

Several canonical lifts (that then lead to canonical lifted couplings using Theorem \ref{thm:shadowCouplings}) can be found in \cite{beiglbock2021shadow}.

In order to prove Theorem \ref{thm:shadowCouplings}, we will need one auxiliary result. Recall the definitions of $x^-_{\mu,\nu},x^+_{\mu,\nu}$ (given by \eqref{eq:x-x+}), of $\eta^-_{\mu,\nu},\eta^0_{\mu,\nu},\eta^+_{\mu,\nu}$ (given by \eqref{eq:mu_decomposition}), and of $\chi^-_{\mu,\nu},\chi^0_{\mu,\nu},\chi^+_{\mu,\nu}$ (given by \eqref{eq:nu_decompisition}).
\begin{lemma}\label{lem:decomposition}
Let $\mu,\nu\in\sP$. For each lift $\hat\mu\in\Pi(\lambda,\nu)$ define
$$
\hat\mu^-_{[0,u]}=\hat\mu_{[0,u]}\lvert_{(-\infty,x^-_{\mu,\nu})},\quad \hat\mu^0_{[0,u]}=\hat\mu_{[0,u]}\lvert_{[x^-_{\mu,\nu},x^+_{\mu,\nu}]},\quad \hat\mu^+_{[0,u]}=\hat\mu_{[0,u]}\lvert_{(x^+_{\mu,\nu},\infty)},\quad\textrm{for all }u\in[0,1].
$$
Then
$$
S^\nu(\hat\mu_{[0,u]})=S^{\chi^-_{\mu,\nu}}(\hat\mu^-_{[0,u]})+S^{\chi^0_{\mu,\nu}}(\hat\mu^0_{[0,u]})+S^{\chi^+_{\mu,\nu}}(\hat\mu^+_{[0,u]}),\quad\textrm{for all }u\in[0,1].
$$
\end{lemma}
\begin{proof}
We use similar arguments as in the proof of Theorem \ref{thm:coupling}.

First, by Lemma \ref{lem:mu_nu_decomposition}, and the fact that $\chi^+_{\mu,\nu}$ is the restriction of $\nu$ to $(x^+_{\mu,\nu},\infty)$ (together with an appropriate amount of mass at $x^+_{\mu,\nu}$), we have that $\eta^+_{\mu,\nu}\leq_{cd}\chi^+_{\mu,\nu}\leq_{cd}\theta^+$ and $\overline{\chi^+_{\mu,\nu}}>\overline{\theta^+}$ for any other possible target law $\theta^+\in\mathcal{T}_{\eta^+_{\mu,\nu},\nu}$ (of $\eta^+_{\mu,\nu}$ in $\nu$). Using Lemma \ref{lem:WOTsuper} we then conclude that
\begin{align}\label{eq:minimalEta}
\inf_{\pi\in\Pi(\eta^+_{\mu,\nu},\chi^+_{\mu,\nu})}\int_\R\left \lvert x-\int_\R y\pi_x(dy)\right\lvert\eta^+_{\mu,\nu}(dx)&=\overline{\eta^+_{\mu,\nu}}-\overline{\chi^+_{\mu,\nu}}\nonumber\\&<\overline{\eta^+_{\mu,\nu}}-\overline{\theta^+}\\&=\inf_{\pi\in\Pi(\eta^+_{\mu,\nu},\theta^+)}\int_\R\left \lvert x-\int_\R y\pi_x(dy)\right\lvert\eta^+_{\mu,\nu}(dx)\nonumber,
\end{align}
and thus $\chi^+_{\mu,\nu}$ is the unique element of $\argmin_{\theta^+\in\mathcal{T}_{\eta^+_{\mu,\nu},\nu}}V(\eta^+_{\mu,\nu},\theta^+)$. It follows that $S^\nu(\eta^+_{\mu,\nu})=\chi^+_{\mu,\nu}$. (Note that in this case $S^\nu(\eta^+_{\mu,\nu})$ reduces to the supermartingale shadow measure of \cite{NutzStebegg.18}.)

Now observe that $\hat\mu^+_{\mu,\nu}\leq\eta^+_{\mu,\nu}\leq_{cd}\chi^+_{\mu,\nu}\leq \nu$, and therefore $\hat\mu^+_{\mu,\nu}\leq\eta^+_{\mu,\nu}\leq_{pcd}\nu$. Hence, by applying \cite[Lemma 3.12]{BayDengNorgilas} we obtain that $\hat\mu^+_{\mu,\nu}\leq_{pcd}S^\nu(\eta^+_{\mu,\nu})$ and
$$
S^\nu(\hat\mu^+_{[0,u]})=S^{S^\nu(\eta^+_{\mu,\nu})}(\hat\mu^+_{[0,u]})=S^{\chi^+_{\mu,\nu}}(\hat\mu^+_{[0,u]})
$$
(where we again use that all the shadow measures are in fact the supermartingale shadow measures of \cite{NutzStebegg.18}). By the associativity of the shadow measure (Lemma \ref{lem:shadow_assoc}) it then follows that
$$
S^\nu(\hat\mu_{[0,u]})=S^{\chi^+_{\mu,\nu}}(\hat\mu^+_{[0,u]})+S^{\nu-S^{\chi^+_{\mu,\nu}}(\hat\mu^+_{[0,u]})}(\hat\mu^-_{[0,u]}+\hat\mu^0_{[0,u]}).
$$

We are left to show that
\begin{equation}\label{eq:remaining}
S^{\nu-S^{\chi^+_{\mu,\nu}}(\hat\mu^+_{[0,u]})}(\hat\mu^-_{[0,u]}+\hat\mu^0_{[0,u]})=S^{\chi^-_{\mu,\nu}}(\hat\mu^-_{[0,u]})+S^{\chi^0_{\mu,\nu}}(\hat\mu^0_{[0,u]}).
\end{equation}

First note that
$$
\eta^-_{\mu,\nu}+\eta^0_{\mu,\nu}\leq_{ci}\chi^-_{\mu,\nu}+\chi^0_{\mu,\nu}\leq \nu-S^{\chi^+_{\mu,\nu}}(\hat\mu^+_{[0,u]}).
$$
Furthermore, $\chi^-_{\mu,\nu}+\chi^0_{\mu,\nu}$ is a restriction of $\nu-S^{\chi^+_{\mu,\nu}}(\hat\mu^+_{[0,u]})$ to $(-\infty,x^+_{\mu,\nu})$ (again, with an appropriate amount of mass at $x^+_{\mu,\nu}$). Hence,
$$
\eta^-_{\mu,\nu}+\eta^0_{\mu,\nu}\leq_{ci}\chi^-_{\mu,\nu}+\chi^0_{\mu,\nu}\leq_{ci}\theta\quad\textrm{and}\quad\overline{\chi^-_{\mu,\nu}+\chi^0_{\mu,\nu}\leq_{ci}}<\overline{\theta}
$$
for any other $\theta\in\mathcal{T}_{\eta^-_{\mu,\nu}+\eta^0_{\mu,\nu},\nu-S^{\chi^+_{\mu,\nu}}(\hat\mu^+_{[0,u]})}$. Using Lemma \ref{lem:WOTsuper}, and similarly as in \eqref{eq:minimalEta}, we conclude that $S^{\nu-S^{\chi^+_{\mu,\nu}}(\hat\mu^+_{[0,u]})}(\eta^-_{\mu,\nu}+\eta^0_{\mu,\nu})=\chi^-_{\mu,\nu}+\chi^0_{\mu,\nu}$.

Furthermore, since $\hat\mu^-_{[0,u]}+\hat\mu^0_{[0,u]}\leq \eta^-_{\mu,\nu}+\eta^0_{\mu,\nu}$, we also have that $\hat\mu^-_{[0,u]}+\hat\mu^0_{[0,u]}\leq \eta^-_{\mu,\nu}+\eta^0_{\mu,\nu}\leq_{pci}\nu-S^{\chi^+_{\mu,\nu}}(\hat\mu^+_{[0,u]})$. Then using a submartingale version of \cite[Lemma 3.12]{BayDengNorgilas} we obtain that 
$$
\hat\mu^-_{[0,u]}+\hat\mu^0_{[0,u]}\leq_{pci} S^{\nu-S^{\chi^+_{\mu,\nu}}(\hat\mu^+_{[0,u]})}(\eta^-_{\mu,\nu}+\eta^0_{\mu,\nu})=\chi^-_{\mu,\nu}+\chi^0_{\mu,\nu}
$$
and
$$
S^{\nu-S^{\chi^+_{\mu,\nu}}(\hat\mu^+_{[0,u]})}(\hat\mu^-_{[0,u]}+\hat\mu^0_{[0,u]})=S^{S^{\nu-S^{\chi^+_{\mu,\nu}}(\hat\mu^+_{[0,u]})}(\eta^-_{\mu,\nu}+\eta^0_{\mu,\nu})}(\hat\mu^-_{[0,u]}+\hat\mu^0_{[0,u]})=S^{\chi^-_{\mu,\nu}+\chi^0_{\mu,\nu}}(\hat\mu^-_{[0,u]}+\hat\mu^0_{[0,u]}).
$$
Then using \eqref{eq:remaining}, it follows that we are left to show that
$$
S^{\chi^-_{\mu,\nu}+\chi^0_{\mu,\nu}}(\hat\mu^-_{[0,u]}+\hat\mu^0_{[0,u]})=S^{\chi^-_{\mu,\nu}}(\hat\mu^-_{[0,u]})+S^{\chi^0_{\mu,\nu}}(\hat\mu^0_{[0,u]}).
$$

Applying the associativity of the shadow measure we have that
$$
S^{\chi^-_{\mu,\nu}+\chi^0_{\mu,\nu}}(\hat\mu^-_{[0,u]}+\hat\mu^0_{[0,u]})=S^{\chi^-_{\mu,\nu}+\chi^-_{\mu,\nu}}(\hat\mu^-_{[0,u]})+S^{\chi^-_{\mu,\nu}+\chi^0_{\mu,\nu}-S^{\chi^-_{\mu,\nu}+\chi^-_{\mu,\nu}}(\hat\mu^-_{[0,u]})}(\hat\mu^0_{[0,u]}),
$$
and thus we need to show that
$$
S^{\chi^-_{\mu,\nu}+\chi^-_{\mu,\nu}}(\hat\mu^-_{[0,u]})=S^{\chi^-_{\mu,\nu}}(\hat\mu^-_{[0,u]})\quad\textrm{and}\quad S^{\chi^-_{\mu,\nu}+\chi^0_{\mu,\nu}-S^{\chi^-_{\mu,\nu}+\chi^-_{\mu,\nu}}(\hat\mu^-_{[0,u]})}(\hat\mu^0_{[0,u]})=S^{\chi^0_{\mu,\nu}}(\hat\mu^0_{[0,u]}).
$$
This can be done by repeating the above arguments, adapted to the submartingale and martingale settings, respectively.
\end{proof}
\begin{proof}[Proof of Theorem \ref{thm:shadowCouplings}]
    For each lift $\hat\mu\in\Pi(\lambda,\mu)$, and the corresponding parametrization $(\hat\mu_{[0,u]})_{u\in[0,1]}$, let $(\hat\mu^-_{[0,u]})_{u\in[0,1]}$, $(\hat\mu^0_{[0,u]})_{u\in[0,1]}$ and $(\hat\mu^+_{[0,u]})_{u\in[0,1]}$ be the families of measures as in Lemma \ref{lem:decomposition}. Note that these families of measures correspond (up to the scaling with total mass) to the lifts of $\eta^-_{\mu}$, $\eta^0_{\mu,\nu}$ and $\eta^+_{\mu,\nu}$, respectively. Then using the results of \cite{beiglbock2021shadow} and \cite{BayDengNorgilas2}, we have that there exists the unique couplings $\hat\pi^-\in\Pi_{Sub}(\eta^-_{\mu,\nu},\chi^-_{\mu,\nu}),\hat\pi^0\in\Pi_{M}(\eta^0_{\mu,\nu},\chi^0_{\mu,\nu}),\hat\pi^+\in\Pi_{Sup}(\eta^+_{\mu,\nu},\chi^+_{\mu,\nu})$ such that, for all $u\in[0,1]$,
    \begin{align}\label{eq:decompProperties}
\int^u_0d\hat\pi^-\in&\Pi_{Sub}(\hat\mu^-_{[0,u]},S^{\chi^-_{\mu,\nu}}(\hat\mu^-_{[0,u]})),\nonumber\\
\int^u_0d\hat\pi^0\in&\Pi_{M}(\hat\mu^0_{[0,u]},S^{\chi^0_{\mu,\nu}}(\hat\mu^0_{[0,u]})),\\
\int^u_0d\hat\pi^+\in&\Pi_{Sup}(\hat\mu^+_{[0,u]},S^{\chi^+_{\mu,\nu}}(\hat\mu^+_{[0,u]})).\nonumber
    \end{align}
    Then, if we define $\hat\pi=\hat\pi^-+\hat\pi^0+\hat\pi^+$, using Lemma \ref{lem:decomposition} we conclude that $\hat\pi$ satisfies the required properties.

    For the uniqueness, if $\tilde\pi$ is another element of $\hat\Pi(\hat\mu,\nu)$ that satisfies the stated properties, then using Lemma \ref{lem:decomposition} and the uniqueness of couplings $\hat\pi^-\in\Pi_{Sub}(\eta^-_{\mu,\nu},\chi^-_{\mu,\nu}),\hat\pi^0\in\Pi_{M}(\eta^0_{\mu,\nu},\chi^0_{\mu,\nu}),\hat\pi^+\in\Pi_{Sup}(\eta^+_{\mu,\nu},\chi^+_{\mu,\nu})$ satisfying \eqref{eq:decompProperties}, we conclude that $\tilde\pi=\hat\pi^-+\hat\pi^0+\hat\pi^+=\hat\pi$.
    \end{proof}
\bibliographystyle{plain}   
\bibliography{references}
	
\end{document}